\documentclass{amsart}

\newtheorem{theorem}{Theorem}[section]

\newtheorem{proposition}{Proposition}[section]
\usepackage{amsmath,amssymb,amsthm}
\theoremstyle{definition}
\newtheorem{definition}[theorem]{Definition}
\newtheorem{example}[theorem]{Example}

\newtheorem{corollary}[theorem]{Corollary}
\theoremstyle{remark}
\newtheorem{remark}[theorem]{Remark}
\usepackage{enumitem}
\numberwithin{equation}{section}
\begin{document}
	
	\title{Weighted Coordinates Poset Block Codes}
	
	
	\author{Atul Kumar Shriwastva}
	\address{{Department of Mathematics, National Institute of Technology Warangal, Hanamkonda, Telangana 506004, India}}
	\email{shriwastvaatul@student.nitw.ac.in}
	\thanks{}
	
	\author{R. S. Selvaraj}
	\address{{Department of Mathematics, National Institute of Technology Warangal, Hanamkonda, Telangana 506004, India}}
	\email{rsselva@nitw.ac.in}
	\thanks{}
	
	\subjclass[2010]{Primary: 94B05, 06A06; Secondary:  15A03}
	
	\keywords{Error block codes, Poset codes, Poset block codes, Weighted coordinates
		poset codes, Weight distribution, Singleton bound, MDS codes}
	
	\date{}
	
	\dedicatory{}		
		\begin{abstract}
			Given $[n]=\{1,2,\ldots,n\}$, a partial order  $\preceq$  on $[n]$, a label map  $\pi : [n] \rightarrow \mathbb{N}$ defined by $\pi(i) = k_i$ with $\sum_{i=1}^{n}\pi (i) = N$,  the direct sum $ \mathbb{F}_{q}^{k_1} \oplus \mathbb{F}_{q}^{k_2}\oplus \ldots \oplus \mathbb{F}_{q}^{k_n} $ of $ \mathbb{F}_q^N $, and a weight function $w$ on  $ \mathbb{F}_q $, we define a poset block metric $d_{(P,w,\pi)}$ on $\mathbb{F}_{q}^{N}$ based on the poset $P=([n],\preceq)$. The metric  $d_{(P,w,\pi)}$ is said to be weighted coordinates poset block metric ($(P,w,\pi)$-metric). It extends the weighted coordinates poset metric ($(P,w)$-metric)  introduced by L. Panek and J. A. Pinheiro and generalizes the poset block metric ($(P,\pi)$-metric) introduced by  M. M. S. Alves et al.
			We determine the complete weight distribution of a $(P,w,\pi)$-space, thereby obtaining it for $(P,w)$-space, $(P,\pi)$-space, $\pi$-space, and $P$-space as special cases. We obtain the Singleton bound for $(P,w,\pi)$-codes and for $(P,w)$-codes as well. 
			In particular, we re-obtain the Singleton bound for  any code with respect to $(P,\pi)$-metric and $P$-metric. Moreover, packing radius and Singleton bound for NRT block codes are found.
		\end{abstract}
		\maketitle

\section{Introduction}
 Niederreiter \cite{hnt} generalized the classical problem of coding theory and  found what could be the largest minimum distance $d$ of any linear code $[n,k,d]$ of length $n$ over the finite field $\mathbb{F}_q$, for any integer $ n > k \geq 1$.       In  \cite{Bru}, Brualdi explored the same problem and introduced metrics based on posets by using partially ordered relations on the set $[n]$, where $[n] = \{1,~2,\ldots,n\}$ represents the coordinate positions of $n$-tuples  in the vector space $\mathbb{F}_q^n$. 
In 2006, block codes of length $N=\pi(1)+  \ldots + \pi(n)$ in $ \mathbb{F}_{q}^N$,  were  introduced by \textit{K. Feng}  \cite{fxh} by using  a label map $\pi$  from $[n] $ to $ \mathbb{N}$ such that $\sum_{i=1}^{n}\pi (i) = N$ and $\mathbb{F}_{q}^{N} = \mathbb{F}_{q}^{\pi(1)}  \oplus \mathbb{F}_{q}^{\pi(2)} \oplus \ldots \oplus \mathbb{F}_{q}^{\pi(n)}$.  
Further,  Alves et al. \cite{Ebc}, introduced  $(P,\pi)$-block codes using partial order relation on the block  positions $[n]$ and subsequently those codes were studied in several research papers (refer \cite{bk}, \cite{bkdn}, \cite{bkdns}, \cite{bkdnsr}, \cite{book},   \cite{nrt block}, \cite{Classification of poset-block }). Block codes have applications in experimental design, high-dimensional numerical integration, and cryptography. A space equipped with  pomset mertics  is a recently introduced one by I. G. Sudha, R. S. Selvaraj \cite{gsrs}  which is  a generalization of Lee space \cite{Leecode}, in particular, and  poset space \cite{Bru}, in general, over $\mathbb{Z}_m$. However, L. Panek \cite{wcps} introduced the  weighted coordinates poset metric very recently (2020) which is a simplified version of the pomset metric that does not use the multiset structure. 
\par We will recall certain basic definitions in order to facilitate the organization of this paper. If $R$ is a ring and $N$ is a positive integer, a map $w : R^N \rightarrow \mathbb{N} \cup \{0\}$ is said to be a weight on $R^N $ if it satisfies the following properties: 
$(a)$  $w(u) \geq 0$;  $u \in R^N $
$(b)$  $w(u) = 0$ iff  $ u = 0$ 
$(c)$  $w(- u)  = w(u) $;  $u \in R^N $ 
$(d)$  $w(u + v)  \leq w(u) + w(v)  $; $u,v \in R^N $.
  \par Given a weight $w$ on $R^N$, if we define a map $d : R^N \times R^N \rightarrow \mathbb{N} \cup \{0\}$ by $d(x, y) \triangleq w(x-y)$, then $d$ is a metric on $R^N$. As the  metric $d$ is determined by a weight, it is invariant by translations: $d(x+z, y+z) = d(x, y)$ for all $x, y, z  \in R^N$. Moreover, the map $w(x) \triangleq d(x, 0) $ is a weight on $R^N$ whenever $d$ is translation-invariant metric. If $w$ is a weight on $R$, then the function $w^N $ defined by $ w^N  \triangleq \sum\limits_{i=1}^{N} w(x_i)$ is a weight on $R^N$ induced by $w$, called additive weight, and the metric $d_{w^N}$ on $R^N$  defined by $ d_{w^N}(x,y)  \triangleq w^N(x-y)$ is an additive metric. 
 \par 
 For example, if $m$ is a  positive integer, the  Lee weight $w_L$ of $a\in \mathbb{Z}_m$ is $\min\{a,  m-a\}$  and the Lee weight of an $n$-tuple $u=(u_1,u_2,\ldots,u_n) \in \mathbb{Z}_m^n$ is $w_L{(x)}   = \sum_{i=1}^{n}{w_L{(u_i)}} $. Lee distance between  $u, v \in \mathbb{Z}_{m}^n$ is  $ d_{L}(u,v) = w_L{(u - v)} $. Support of an $n$-tuple $u=(u_1,u_2,\ldots,u_n) \in \mathbb{F}_q^n$ is defined to be the set $ supp(u)= \{i \in [n] : u_i \neq 0  \}$ and the Hamming weight of $u$ is $w_H(u)= |supp(u)|$. Here both Lee weight and Hamming weight are additive.
 \par Let $P=([n],\preceq)$ be a poset.  An element $j \in A \subseteq P$ is said to be a maximal element of $A$ if there is no $i \in A$ such that $j \preceq i$. An element $j \in A \subseteq P$ is said to be a minimal element of $A$ if there is no $i \in A$ such that $i \preceq j$.  A subset $I$ of $P$ is said to be an ideal if $j \in I$ and $i \preceq j$ imply $i \in I$. For a subset $A$  of $P$, an ideal generated by $A$ is the smallest ideal containing $A$ and is denoted by $\langle A \rangle$.  
Poset weight or $P$-weight of $u \in \mathbb{F}_q^n$   is $ w_P (u)= |\langle supp(u) \rangle|$ and $P$-distance between $u,v \in \mathbb{F}_q^n$ is $ d_{P} (u,v)=  w_{P}   (u-v)$.  
\par Through a label map  $\pi  : [n] \rightarrow \mathbb{N} $ defined as $\pi (i) =k_i $ such that $ \sum\limits_{i=1}^{n} \pi (i) = N  $ and considering  $ \mathbb{F}_q^N $ as the direct sum  $ \mathbb{F}_{q}^{k_1} \oplus \mathbb{F}_{q}^{k_2} \oplus \cdots \oplus \mathbb{F}_{q}^{k_n} $, one can express $x \in \mathbb{F}_q^N$  uniquely as  $x= x_1 \oplus x_2 \oplus \cdots \oplus x_n $ with  $x_i = (x_{i_1},x_{i_2},\ldots,x_{i_{k_i}}) \in \mathbb{F}_{q}^{k_i} $ and can define the block-support  \cite{Ebc},  \cite{fxh} or $\pi$-support of $x$ as $ supp_\pi(x)= \{i \in [n] : 0 \neq  x_i  \in \mathbb{F}_q^{k_i}\}$. 
Now, $\pi$-weight of $x$ is defined as $ w_\pi (x)= |supp_\pi(x)|$ and $\pi$-distance 
between  $x,y \in \mathbb{F}_q^N$ is $ d_\pi (x,y)=  w_\pi (x-y)$. Note that if $k_i =1 ~ \forall ~ i$, the $\pi$-weight and $\pi$-distance becomes Hamming weight and Hamming distance respectively. Now, the poset block weight or  $ (P,\pi)$-weight of $x$ is defined as $ w_{(P,\pi)} (x) \triangleq  |\langle supp_\pi(x) \rangle|$ and $(P,\pi)$-distance between $x,y \in \mathbb{F}_q^N$ is defined as $ d_{(P,\pi)} (x,y) \triangleq  w_{(P,\pi)}   (x-y)$. 
\par By assigning a weight for each coordinate position of a vector  $u \in \mathbb{F}_{q}^{n} $, L. Panek \cite{wcps} introduced what is called as  weighted coordinates poset weight. For this, let $u \in \mathbb{F}_{q}^{n} $, 
 $I_{u}^P = \langle supp(u) \rangle $  and $M_{u}^P$ be the set of all maximal elements in $I_{u}^P$.  If $w$ is  a weight defined on $\mathbb{F}_q$ and $M_w = \max \{w(\alpha) :   \alpha \in \mathbb{F}_q \}$, then the weighted coordinates poset weight or $(P,w)$-weight of  $u $  is  defined as 
$ 	w_{(P,w)}(u) = \sum\limits_{i \in M_{u}^P} w{(u_i)} + \sum\limits_{i \in {I_{u}^{P} \setminus M_{u}^P}} M_w $. The weighted coordinates poset distance or 
 $(P,w)$-distance between   $u,v \in \mathbb{F}_q ^n$ is defined as $ d_{(P,w)}(u,v) \triangleq w_{(P,w)}(u-v) $. If $w(\alpha) = 1$ $\forall$  $ \alpha \in \mathbb{F}_q \setminus 0$, then the $(P,w)$-weight of  $u $  becomes the $P$-weight of  $u $.
\par  In this paper, we introduce weighted coordinates poset block metric (or  $(P,w,\pi)$-metric) for codes of length $N$ over the field $\mathbb{F}_q$. We will  define it in Section ${2}$, by attaching a weight $\tilde{w}^{k_i}$ to each block of length $k_i$ of a codeword where $\tilde{w}^{k_i}$ is a weight on $\mathbb{F}_q^{k_i} $ that depends on a weight  $w$ on $\mathbb{F}_q$.  In Section ${3}$, we will determine the   complete weight distribution of $\mathbb{F}_q ^{N}$   with respect to $(P,w,\pi)$-metric, in general, thereby, facilitating the weight distribution  with respect to  $(P,w)$-metric, $(P,\pi)$-metric, $\pi$-metric, $P$-metric as particular cases. Moreover, we determined  cardinality of $r$-ball with respect to $(P,w,\pi)$-metric and  $(P,w)$-metric.  The Singleton bounds of a code  with respect to $(P,w,\pi)$-metric, $(P,w)$-metric, are  established  in Section ${4}$ and are  re-obtained for any code with respect to $(P,\pi)$-metric, and $P$-metric. In particular, when all the blocks have the same dimension, a necessary condition for a code to be MDS is also found. In Section ${5}$, complete weight distribution, Packing radius, and Singleton bound for NRT block spaces are determined by considering $P$ to be a chain. 
\section{Weighted Coordinates Poset Block Metrics}
	To start with, let  $w$ be a weight on $\mathbb{F}_q$ and $M_w = \max \{w(\alpha) :   \alpha \in \mathbb{F}_q \}$. For a $k \in \mathbb{N}$, and a $v=(v_1, v_2,\ldots,v_k) \in  \mathbb{F}_q^{k} $, we define  $\tilde{w}^{k}{(v)} = \max \{ w(v_i) : 1 \leq i \leq k\}$. Clearly,   $\tilde{w}^{k}$ is a weight on $\mathbb{F}_q^{k} $ induced by the weight $w$. On 	$ \mathbb{F}_{q}^{k_i} $, $1 \leq i \leq n$, we call   $\tilde{w}^{k_i}$, a block weight. Note that  $\tilde{w}^{k}$ is not an additive weight and the metric $d_{\tilde{w}^{k}}$ induced by ${w}$-weight on $\mathbb{F}_q^{k} $ defined as $d_{\tilde{w}^{k}} (u,v) \triangleq \tilde{w}^{k}(u-v)$  is not an additive metric.
\par
  Considering the finite set $[n]=\{1,2,\ldots,n\}$ with $\preceq $ to be a partial order,  the pair $P=([n],\preceq)$ is  a poset.  With a label map  $\pi  : [n] \rightarrow \mathbb{N} $ defined  as $\pi (i) =k_i $ in the previous section such that $ \sum\limits_{i=1}^{n} \pi (i) = N  $, a positive inetger, we have $	\mathbb{F}_{q}^{N} = \mathbb{F}_{q}^{k_1}  \oplus \mathbb{F}_{q}^{k_2} \oplus \ldots \oplus \mathbb{F}_{q}^{k_n} $. Thus, if  $x \in  \mathbb{F}_{q}^{N} $ then $x= x_1 \oplus x_2 \oplus \cdots \oplus x_n $ with  $x_i = (x_{i_1},x_{i_2},\ldots,x_{i_{k_i}}) \in \mathbb{F}_{q}^{k_i} $. Let $I_{x}^{P,\pi} =  \langle supp_{\pi}(x) \rangle $ be the ideal generated by the $\pi$-support of $x$ and $M_{x}^{P,\pi}$ be the set of all maximal elements in $I_{x}^{P,\pi}$.
 \begin{definition}[$(P,w,\pi)$-weight]
   \sloppy{Given a poset $P=([n],\preceq)$, a weight $w$ on $\mathbb{F}_q $, a label map  $\pi $ with a block weight $\tilde{w}^{k_i}$ on  $\mathbb{F}_q^{k_i} $ induced by $w$, the 
   	weighted coordinates poset block weight or 
 	$(P,w,\pi)$-weight of  $x \in \mathbb{F}_q^N$ is defined as
 	\begin{equation*}
 		w_{(P,w,\pi)}(x) \triangleq \sum\limits_{i \in M_{x}^{P,\pi}} \tilde{w}^{k_i}{(x_i)} + \sum\limits_{i \in {I_{x}^{P,\pi} \setminus M_{x}^{P,\pi}}} M_w
 	\end{equation*} 
 The  $(P,w,\pi)$-distance between two vectors  $x,y \in \mathbb{F}_q ^N$ is defined as:
 $ d_{(P,w,\pi)}(x,y) \triangleq  w_{(P,w,\pi)}(x-y)$.}
 \end{definition} 

\begin{theorem}\label{t1}
	The $(P,w,\pi) $-distance is  a metric on $\mathbb{F}_{q}^N$.
\end{theorem}
\begin{proof}
	Though the proof follows similar pattern as in \cite{wcps}, we provide it to facilitate familiarity of the	terminologies we deal with. As  $w$ and $\tilde{w}^{k_i}$ are the weights on $ \mathbb{F}_{q} $ and $ \mathbb{F}_{q}^{k_i} $ respectively,  it is clear that  $	d_{(P,w,\pi)}(x,y) \geq0  $ for every $x, y \in \mathbb{F}_{q}^N$.  Also, $	 d_{(P,w,\pi)}(x,y) = 0 $ iff $I_{x-y}^{P,\pi} = \phi $ iff $x = y $. As $ supp_{\pi} (-x) = supp_{\pi} (x)$ for any $x \in \mathbb{F}_{q}^{N}$, we have  
  $	d_{(P,w,\pi)}(x,y) = d_{(P,w,\pi)}(y,x)$ for all $x , y \in \mathbb{F}_{q}^{N}$. Now  $w_{(P,w,\pi)}(x+y) = \sum\limits_{i \in M_{x+y}^{P,\pi}} \tilde{w}^{k_i}{(x_i + y_i)} + \sum\limits_{i \in {I_{x +y}^{P,\pi} \setminus M_{x+y}^{P,\pi}}} M_w  $. As  $\tilde{w}^{k_i}{(x_i + y_i)}  \leq \tilde{w}^{k_i}{(x_i)}  +  \tilde{w}^{k_i}	{(y_i)}$ and $M_{x+y}^{P,\pi} \subseteq M_{x}^{P,\pi} \cup M_{y}^{P,\pi}$, we have
\begin{equation}\label{e1}
	\sum\limits_{i \in M_{x+y}^{P,\pi}} \tilde{w}^{k_i}	{(x_i + y_i)} \leq \sum\limits_{i \in M_{x}^{P,\pi}} \tilde{w}^{k_i}	{(x_i)}  +  \sum\limits_{i \in M_{y}^{P,\pi}} \tilde{w}^{k_i}	{(y_i)}
\end{equation}  Since 
$ I_{x +y}^{P,\pi} \setminus M_{x+y}^{P,\pi} \subseteq (I_{x }^{P,\pi} \setminus  M_{x}^{P,\pi}) \cup (I_{y}^{P,\pi} \setminus M_{y}^{P,\pi})$ we will have 
\begin{equation}\label{e2}
	\sum\limits_{i \in {I_{x +y}^{P,\pi} \setminus M_{x+y}^{P,\pi}}} M_w \leq \sum\limits_{i \in {I_{x}^{P,\pi} \setminus M_{x}^{P,\pi}}} M_w + \sum\limits_{i \in {I_{y}^{P,\pi} \setminus M_{y}^{P,\pi}}} M_w
\end{equation} 
From  (\ref{e1}) and (\ref{e2}), we have
\begin{equation*}
	w_{(P,w,\pi)}(x+y) \leq 	w_{(P,w,\pi)}(x)  + 	w_{(P,w,\pi)}(y) 
\end{equation*} 
Hence $	d_{(P,w,\pi)}(x,y) =  w_{(P,w,\pi)}(x-z+z-y) \leq w_{(P,w,\pi)}(x-z) +  w_{(P,w,\pi)}(z-y) = d_{(P,w,\pi)}(x,z) + 	d_{(P,w,\pi)}(z,y)  $, for $x,y,z \in \mathbb{F}_{q}^N$.
\end{proof}
Hence, $d_{(P,w,\pi)}$ defines a metric on $ \mathbb{F}_{q}^N $ called as \textit{weighted cordinates poset block metric} or $ (P,w,\pi) $-metric. The pair $ (\mathbb{F}_{q}^N,~d_{(P,w,\pi)} )$ is said to be a  $(P,w,\pi)$-space.
\par  \sloppy{A $ (P,w,\pi) $-block code  $\mathbb{C} $ of length $N$ is a subset of $ (\mathbb{F}_{q}^N,~d_{(P,w,\pi)} )$-space  and 
$ d_{(P,w,\pi)}\mathbb{(C)} = min \{  d_{(P, w,\pi)} {(c_1, c_2)}: c_1, c_2 \in \mathbb{C} \} $ gives the minimum distance of  $\mathbb{C}$. 
If $\mathbb{C}$ is a linear $ (P,w,\pi) $-block code, then   
$ d_{(P,w,\pi)}\mathbb{(C)} = min \{ w_{(P,w,\pi)}(c) : 0 \neq c \in \mathbb{C} \} $.
 As $ w_{(P,w,\pi)}(v) \leq n M_w$ for any $v \in \mathbb{F}_{q}^N $, the minimum distance of a linear code   $\mathbb{C} $ is bounded above by $ n M_w $.}
\par Though it is a more general fact to say that  the $(P,w,\pi) $-distance is  a metric on a free module $R^N$ over a commutative ring $R$ with identity, we focus our discussion by taking $R = \mathbb{F}_{q}$ almost everywhere. The results obtained in the forthcoming sections for 
$\mathbb{F}_{q}^N $ or $\mathbb{Z}_{m}^N $ are even valid for any  free module $R^N$ over the  commutative ring $R$ with identity.
\par We remark that the $(P,w,\pi)$-distance is a metric which includes several classic metrics of coding theory.
\begin{remark}\label{becomes} 
	Now we shal see how or in what manner the weighted coordinates poset block metric generalizes the various poset metrics arrived at so for.
	\begin{enumerate}[label={(\roman*)}]
		\item If $k_i = 1 $ for every $i\in [n]$ then the block support of $v \in \mathbb{F}_q^N$ becomes the usual support of $v \in \mathbb{F}_q^n$ 
		and $(P,w,\pi)$-weight of $v \in \mathbb{F}_q^N$ becomes $(P,w)$-weight of $v \in \mathbb{F}_q^n$. Thus, the $(P,w,\pi)$-space becomes the $(P,w)$-space (as in  \cite{wcps}).
		\item 	If $k_i = 1 $ for every $i\in [n]$ and $w$ is the Hamming weight on $\mathbb{F}_q$ so that  $w(\alpha)=1$ for $0 \neq \alpha \in \mathbb{F}_q$,  then for any  
		$ v = v_1 \oplus v_2 \oplus \ldots \oplus v_n \in \mathbb{F}_q^N$, 
		\begin{align*}
			 \tilde{w}^{k_i}{(v_i)} =
			\max\{w(v_{i_t}) : 1 \leq t \leq k_i\} = \left\{ 
			\begin{array}{ll}
				0, &  \text{if} \ v_{i} = 0\\
				1, &   \text{if} \ v_{i} \neq 0
			\end{array}	  \right\} =  w_{H}(v_{i}),
		\end{align*}
		the Hamming weight of $v_i$ in $\mathbb{F}_q$.
		Thus, the	 $(P,w,\pi)$-space becomes the poset space or $P$-space (as in \cite{Bru}).
		\item 	If $w$ is the Hamming weight  on $\mathbb{F}_q$, then $w(\alpha)=1$ for $0 \neq \alpha \in \mathbb{F}_q$.  Let $v \in \mathbb{F}_q^N$, $v_i = (v_{i_1},v_{i_2}, \ldots,v_{i_{k_i}}) \in \mathbb{F}_{q}^{k_i} $ and $v_{i_{t}} \in  \mathbb{F}_{q}$, $1 \leq t \leq k_i$, so that  $ 	 \tilde{w}^{k_i}{(v_i)} = \max\{w(v_{i_t}) : 1 \leq t \leq k_i\}   = \left\{ 
		\begin{array}{ll}
			0, &  \text{if} \ v_{i} = 0\\
			1, &    \text{if} \ v_{i} \neq 0
		\end{array}	  \right\}  $. Hence,
		\begin{align*}
			w_{(P,w,\pi)}(v) = | M_{v}^{P,\pi} |  + | {I_v}^{P,\pi} \setminus M_{v}^{P,\pi} | = | I_{v}^{P,\pi} | = w_{(P,\pi)}(v). 
		\end{align*} 
		Thus, the  $(P,w,\pi)$-space becomes the $(P,\pi)$-space  (as in \cite{Ebc}).
		\item 	If $w$ is the Hamming weight on $\mathbb{F}_q$ so that  $w(\alpha)=1$ for $0 \neq \alpha \in \mathbb{F}_q$, and $P$ is an antichain, then for any $v \in \mathbb{F}_q^N$, 
		$ 	 \tilde{w}^{k_i}{(v_i)} =  \left\{ 
		\begin{array}{ll}
			0, &  \text{if} \ v_{i} = 0\\
			1, &   \text{if} \ v_{i} \neq 0
		\end{array}	  \right\} $, $I_v^{P,\pi} = supp_{\pi} (v)$ and hence,
		$		w_{(P,w,\pi)}(v) = | M_{v}^{P,\pi} |  + | {I_v}^{P,\pi} \setminus M_{v}^{P,\pi} | = | supp_{\pi}(v) |= w_{\pi}(v) 
		$.
		Thus, the  $(P,w,\pi)$-space becomes the $\pi$-space or $( \mathbb{F}_q^N,d_{\pi})$-space  (as in \cite{fxh}). 
	\end{enumerate} 
\end{remark}
Throughout  the paper, let $ \mathcal{I}(P) = \{ I \subseteq P: I  \text{~is an ideal} \}$ denote the collection of all ideals in $P$.  Let $\mathcal{I}^{i}$ be its sub-collection of all ideals whose cardinality is $i$. Let 
$\mathcal{I}_{j}^{i}$ denote the collection of all ideals $I \in \mathcal{I}(P)$ with cardinality $i$ having $j$ maximal elements. Clearly, $\mathcal{I}_{j}^{i} \subseteq \mathcal{I}^{i}$ and $\cup_{j=1}^{i} \mathcal{I}_{j}^{i} =\mathcal{I}^{i} $, for $1 \leq i \leq n$. Given an ideal $I \in \mathcal{I}_{j}^{i}$, let $Max(I) = \{i_1, i_2, \ldots, i_j\}$ be the set of all maximal elements of $I$ and $I \setminus Max(I) = \{ l_1, l_2,\ldots, l_{i-j}\}$ be the set of all non-maximal elements of $I$, if any. Here $i,j$ are integers such that $ 1 \leq i \leq n $ and $ 1 \leq j \leq i$. 
\par  The $(P,w,\pi)$-ball centered at a point $y \in \mathbb{F}_q^N$ with radius $r$ is the set $	B_{r}(y)=\{x \in \mathbb{F}_{q}^{N} : d_{(P,w,\pi)}(y,x) \leq r\}$. The $(P,w,\pi)$-sphere centered at $y$  with radius $r$ is  $S_{r}(y)=\{x \in \mathbb{F}_{q}^{N} : d_{(P,w,\pi)}(y,x) = r\}$. $B_r(y)$ and $S_r(y)$ are also called as $r$-ball and $r$-sphere respectively. As the $d_{(P,w,\pi)}$-metric is translation invariant, we have  $B_{r}(y) = y + B_{r}(0)$ where $0 \in \mathbb{F}_q^N$. Clearly, $ | B_{r}(0) | =  1 + \sum\limits_{t=1}^{r} | S_{t}(0) |$ and $ | S_{t}(0) |$ is equal to the number of $x \in \mathbb{F}_{q}^{N}$ such that $w_{(P,w,\pi)}(x)=t$.   Therefore, to determine the cardinality of an $r$-ball, first we shall  find the $(P,w,\pi)$-weight distribution of $\mathbb{F}_{q}^{N}$. 
\section{Weight Distribution of $(P,w,\pi)$-space}
For a weight $w$ on $ \mathbb{F}_{q} $,  the sets  $D_r = \{\alpha \in \mathbb{F}_q : w(\alpha) = r\}$, for $0 \leq r \leq M_w $ partition $ \mathbb{F}_{q} $ according to the distribution of weights $w$ of elements in $ \mathbb{F}_{q} $. Recall that,   $ \tilde{w}^{k} $ is a weight on   $ \mathbb{F}_{q}^{k} $ defined as  $ \tilde{w}^{k}{(u)}=\max\{w(u_{i}) : 1 \leq i \leq k\}$ for $u=(u_{1},u_{2},~\ldots,u_{k}) \in \mathbb{F}_{q}^{k} $.  For each $0 \leq r \leq M_w $, we shall define:
\begin{align*}
	D_{r}^{k}= \{u=(u_{1},u_{2},\ldots,u_{{k}}) \in \mathbb{F}_{q}^{k}  :  \tilde{w}^{k}(u)=r  \}
\end{align*}
to be a subset of $ \mathbb{F}_{q}^{k} $ so that $D_{r}^{k}$ $(0 \leq r \leq M_w )$ partitons $ \mathbb{F}_{q}^{k} $ according to the
distribution of weights $\tilde{w}^{k}$ of elements in 
$ \mathbb{F}_{q}^{k} $. 
Thus, we have the following:
\begin{proposition} \label{D_r^j}
	For a given positive integer $ k $, $|D_r^{k}|=( \sum\limits_{i=0}^{r} |D_{i}|)^{k}-( \sum\limits_{i=0}^{r-1} |D_{i}|)^{k}$ for $0 \leq r \leq M_w $. In particular, 
	$|D_{M_w}^{k}|=q^{k}-(q- |D_{M_w} |)^{k}$.
\end{proposition}
\begin{proof}
	Let $u=(u_{1},u_{2},\ldots,u_{{k}}) \in D_{r}^{k} $ so that $ \tilde{w}^{k}(u)=r$. This means, at least one component of $u$  has  its maximum weight as  $r$. The number of  such $u   \in \mathbb{F}_{q}^{k} $ with $l$ components having  $r$ as their $w$-weight  is $	{k \choose l}|D_{r}|^{l}(1+|D_{1}|+ |D_{2}|+ \ldots + |D_{r-1} |)^{k-l} $. Thus,
	\begin{align*}
		|D_r^{k}|&= \sum\limits_{l=1}^{k} {k \choose l} |D_{r}|^{l} (1+|D_{1}|+ |D_{2}|+ \ldots + |D_{r-1} |)^{k-l}  \\
		&=( 1+ |D_{1}|+ |D_{2}|+ \ldots +  |D_{r}| )^{k}-  ( 1+|D_{1}|+ |D_{2}|+ \ldots + |D_{r-1} |)^{k}  \\
		&= ( \sum\limits_{i=0}^{r} |D_{i}|)^{k}-( \sum\limits_{i=0}^{r-1} |D_{i}|)^{k}
	\end{align*}
	where $D_r = \{\alpha \in \mathbb{F}_q : w(\alpha) = r\}$.
\end{proof}
The above Proposition \ref{D_r^j} gives the weight distribution of $\mathbb{F}_{q}^{k}$ by considering $ \tilde{w}^{k}$ as weight.
\par Next, we shall proceed on to determine the $(P,w,\pi)$-weight distribution of $\mathbb{F}_{q}^{N}$. 
For each $1 \leq r \leq nM_w$, let $A_r = \{x= x_{1} \oplus x_{2} \oplus \ldots \oplus x_{n} \in  \mathbb{F}_{q}^{N}  : w_{(P,w,\pi)}(x)=r  \}$ and $A_0 = \{ \bar{0}\}$. It is clear that
$x \in A_r$ iff $I_{x}^{P,\pi} = \langle supp_{\pi} (x) \rangle \in \mathcal{I}_j^{i}$ for some $j \in \{1,2,\ldots,n\}$ such that $i \geq j$.
It is also easy to observe the following: 
\begin{itemize}
	\item If $w_{(P,w,\pi)}(x) \leq M_w$  then $I_{x}^{P,\pi} \in \mathcal{I}_j^{j}$ for some $j  \in \{1,2,\ldots,n\}$.
	\item If  $w_{(P,w,\pi)}(x) \geq M_w $  then  $I_{x}^{P,\pi}  \in \mathcal{I}_j^{i}$ for some $j  \in \{1,2,\ldots,n\}$ such that $i \geq j$.
\end{itemize}
 Now, for an $x \in A_r$ with $x= x_1 \oplus x_2 \oplus \cdots \oplus x_n $ and  $x_i = (x_{i_1},x_{i_2},\ldots,x_{i_{k_i}}) \in \mathbb{F}_{q}^{k_i} $, we have 
 \begin{align*}
 w_{(P,w,\pi)}(x) &= \sum\limits_{i \in M_{x}^{P,\pi}}  \tilde{w}^{k_i}{(x_i)} +   \sum\limits_{i \in {I_x}^{P,\pi} \setminus M_{x}^{P,\pi}} M_w \\ &=  \tilde{w}^{k_{i_1}}(x_{i_1})+  \tilde{w}^{k_{i_2}}(x_{i_2})  + \ldots+  \tilde{w}^{k_{i_j}}(x_{i_j})+(i-j)M_w
 \end{align*}
  where $M_x^{P,\pi}= \{i_1, i_2, \ldots, i_j\}$. 
Then, $ \tilde{w}^{k_{i_1}}(x_{i_1}) + \tilde{w}^{k_{i_2}}(x_{i_2})+\ldots + \tilde{w}^{k_{i_j}}(x_{i_j}) =r-(i-j)M_w$. As the  weight is varying on the block positions with respect to maximal elements but fixed on block positions with respect to non-maximal elements, we need to choose $j$ number of blocks $x_{i_s} \in  \mathbb{F}_{q}^{k_{i_s}}  $ such that  $\sum\limits_{s=1}^{j} \tilde{w}^{k_{i_s}} (x_{i_s}) =r-(i-j)M_w $.  Thus, to get the  number of  vectors in $A_r$, we first  define the  partition of the number  $r-(i-j)M_w $ in order to have numbers like  $\tilde{w}^{k_{i_1}} (x_{i_1})$, $\tilde{w}^{k_{i_2}} (x_{i_2}), \ldots$, $\tilde{w}^{k_{i_s}} (x_{i_s})$  as its $j$ parts.  
\begin{definition}[{Partition of $r$}]
	For any positive integer $r$, by a partition of $r$ we mean a finite non-increasing sequence of positive integers $b_1,b_2,\ldots,b_j $ such that $b_1+b_2+\dots+ b_j =r $ and is denoted by $(b_1,b_2,\ldots,b_j)$. $b_1,b_2,\ldots,b_j $ are called as parts of the partition  and they need not be distinct.
	As the number of parts in a partition is restricted to at most $n$,  such a partition is termed as an $n$-part partition. Let $PRT[r]$ denote the set of all  $n$-part partitions of $r$ such that each part in $r$ does not exceed $M_w$ i.e. $PRT[r]=$
	\begin{align*}
		\{(b_1,b_2,\ldots,b_j) : b_1+b_2+\dots+ b_j =  r; ~ 1 \leq b_i \leq M_{w} \ \text{and} \ 1 \leq j \leq n\}.
	\end{align*} 
\end{definition}
Since we need to partition $r-(i-j)M_w $, and $ \tilde{w}^{k_{i_1}}(x_{i_1}) + \tilde{w}^{k_{i_2}}(x_{i_2})+\ldots + \tilde{w}^{k_{i_j}}(x_{i_j}) =r-(i-j)M_w$,  the number of parts  is restricted to at most $n-(i-j)$. Hence, to determine the parts of $r-(i-j)M_w$, we define: 
\begin{align*}
PRT_{i-j}[r]=	\{(b_1,b_2,\ldots,b_t) : b_1+b_2+\dots+ b_t =  r- (i-j)M_w;  1 \leq b_s \leq M_{w} \\  \text{ and }   1 \leq t \leq n-i+j\}
\end{align*}
where $i-j$ is the number of non-maximal elements in a given ideal $I \in \mathcal{I}_{j}^{i}$. Note that $PRT[r]=PRT_0[r]$.
\begin{example}\label{example1}
	Let $r=11$, $M_w =3$ and $n=5$, then the $5$-part partitions of  $11$ are given as
	$	PRT_0[11] = \{(3,3,3,2),(3,3,3,1,1), (3,3,2,2,1),(3,2,2,2,2)\}$. 
	Moreover, when $i=5$, $j=3$, by considering an ideal $I \in \mathcal{I}_3^5$, we have $r-(i-j)M_w=5$,  $n-(i-j)=3$, and hence $	PRT_2[11] = \{(3,2),(3,1,1), (2,2,1)\}$.	
\end{example}
	Let  $b= (b_1,b_2,\ldots,b_j) \in PRT_{i-j}[r]$ and  $I \in \mathcal{I}_{j}^{i}$. As the  weight is varying on the $j$ maximal block positions but fixed on the $i-j$ non-maximal block positions, we need to choose $j$ number of $x_{i_s} \in  \mathbb{F}_{q}^{k_{i_s}}  $ such that  $\sum\limits_{s=1}^{j} \tilde{w}^{k_{i_s}} (x_{i_s}) =r-(i-j)M_w $. 
Parts $b_{i}$'s of $b$ do help in finding suitable  $x_{i} \in  \mathbb{F}_{q}^{k_{i}}  $  that  have weight   $ \tilde{w}^{k_i} (x_{i}) = b_{i}$ so that all these $x_{i} $'s will occur only in the positions of maximal elements of  $I$. Since weights in the positions of non-maximal elements are fixed as $M_w$,  any  $x_{i} \in  \mathbb{F}_{q}^{k_{i}}  $ can occcur in those non-maximal positions. But in the remaining $n-|I|$ positions $l$,  $\bar{0} \in  \mathbb{F}_q^{k_l}$ will  occur. Thus, for a $b= (b_1,b_2,\ldots,b_j) \in PRT_{i-j}[r]$ and  $I \in \mathcal{I}_{j}^{i}$, the set $A =  $
	\begin{align*}
		\bigg\{ 
			{x= x_{1}  \oplus \cdots \oplus  x_{n} \in  \mathbb{F}_{q}^{N} :  x_s =} \left\{ 
		\begin{array}{ll}
				\bar{0} \in \mathbb{F}_q^{k_s},  & \text{for} ~s \in [n] \setminus I \\
				x_s \in	D_{b_{s}}^{k_s}, & \text{for}  \ s \in Max(I) \\
				x_s	\in \mathbb{F}_q^{k_s}, & \text{for}  \ s \in I  \setminus Max(I)
			\end{array}	\right\}, {1 \leq s \leq n}  \bigg\}
\end{align*}
gives the set of those $N$-tuples $x= x_{1} \oplus x_{2} \oplus \ldots \oplus x_{n} \in \mathbb{F}_{q}^{N}$ each of whose $(P,w,\pi)$-weight is $r$, $ \langle supp_{\pi} (x) \rangle = I$ and $ \tilde{w}^{k_s} (x_{s}) = b_{s}$ for $s \in Max(I)$. Such a set $A$ of vectors of $(P,w,\pi)$-weight $r$ can be obtained for each arrangement of a $b \in PRT_{i-j} [r]$ for a given $I \in \mathcal{I}_j^{i} $. 
\begin{example} \label{ex1}
	Let $\mathbb{Z}_{7}^{13}=\mathbb{Z}_{7}^2 \oplus \mathbb{Z}_{7}^3 \oplus \mathbb{Z}_{7}^4 \oplus \mathbb{Z}_{7}^2 \oplus \mathbb{Z}_{7}^2$. Consider $w$ as the Lee weight on $\mathbb{Z}_{7}$. Let $\preceq $ be a partial order relation  on the set $[5]=\{1,2,3,4,5\}$ such that $1 \preceq 2$. Here $I = \{1,2,4\}$  is the ideal of cardinality $3$ with $2$ maximal elements. $Max(I)= \{2,4\}$, $I \setminus Max(I)= \{1\}$ and $[n] - I =\{3,5\}$. As $M_w = 3$, 
	$	PRT_1[8] = \{(3,2),(3,1,1), (2,2,1),(2,1,1,1)\}$. Now, for $b=(3,2) \in  PRT_1[8] $ and  $I = \{1,2,4\}$, at the positions of maximal elements  $Max(I)= \{2,4\}$, $x_2 \in \mathbb{Z}_{7}^3$ such that  $ \tilde{w}^{k_2} (x_{2}) = 3$, and $x_4 \in \mathbb{Z}_{7}^2$ such that $ \tilde{w}^{k_4} (x_{4}) = 2$ will occur. At non-maximal positions of $I$ (here it is $ \{1\}$),  any $x_1 \in \mathbb{Z}_{7}^2$ will occur. In the remaining positions $[n] - I =\{3,5\}$, $x_3 =\bar{0} \in \mathbb{Z}_{7}^4$ and  $x_3 =\bar{0} \in \mathbb{Z}_{7}^2$ will occur. Thus, when  $I = \{1,2,4\}$ and  $b=(3,2) \in  PRT_1[8] $, 
	\begin{equation*}
		\begin{split}
			A=	\{x=x_1 \oplus x_2 \oplus \bar{0} \oplus x_4 \oplus \bar{0}: x_1 \in \mathbb{Z}_{7}^2, ~ x_2 \in \mathbb{Z}_{7}^3 & \text{ with} \  \tilde{w}^{k_2} (x_{2}) = 3, \\ & x_4 \in \mathbb{Z}_{7}^2 ~\text{with}  ~   \tilde{w}^{k_4} (x_{4}) = 2 \}
		\end{split}
	\end{equation*}
	gives the set of vectors that have  $(P,w,\pi)$-weight  $3+3+2=8$. 
	And,  if we take $(2,3)$ which is an arrangement of parts of $b=(3,2)$ then 
	\begin{equation*}
		\begin{split}
			B=	\{x=x_1 \oplus x_2 \oplus \bar{0} \oplus x_4 \oplus \bar{0} : x_1 \in \mathbb{Z}_{7}^2, ~ x_2 \in \mathbb{Z}_{7}^3 & \text{ with}~  \tilde{w}^{k_2} (x_{2}) = 2, \\   &  x_4 \in \mathbb{Z}_{7}^2 ~\text{with}~   \tilde{w}^{k_4} (x_{4}) = 3 \}
		\end{split}
	\end{equation*}
	also gives the set of vectors that have $(P,w,\pi)$-weight $3+2+3=8$ when  $I = \{1,2,4\}$.
\end{example}
\begin{definition}[{Arrangement of $b$}]
	Given an $n$-part partition $ b= (b_1,b_2,\ldots,b_j) \in PRT_{i-j}[r] $ with $j$ parts, an arrangement of $b$ is a $j$-tuple   $ (b_{t_1},b_{t_2},\ldots,b_{t_j}) $ such that $t_i \in \{1,2,\ldots,j\}$ are distinct. Thus, $ARG[(b_1,b_2,\ldots,b_j)] =$
	\begin{align*}
		\{ (b_{t_1},  b_{t_2},  \ldots , b_{t_j}): t_1, t_2, \ldots,t_j \in \{1,2,\ldots,j\},~ t_i \neq t_k~ \text{for} ~ i \neq k \}
	\end{align*}
	denote the set of all such arrangements of a given $b= (b_1,b_2,\ldots,b_j) \in PRT_{i-j}[r] $.  \par Partition of a weight $r$ (or  $r-(i-j)M_w$, as the case may be) and arrangement of their parts play a vital role in determining the weight distribution of a $(P,w,\pi)$-space.
\end{definition}
\begin{example}
	Continuing from the Example \ref{example1} with $r=11$, $M_w =3$ and $n=5$, we can see that the arrangements of  $b= (3,3,3,2) \in PRT_0[11] $ are given by
$		ARG[(3,3,3,2)] = \{ (3,3,3,2),(3,3,2,3), (3,2,3,3),(2,3,3,3)\} 		$ 		
	and the arrangements of  $b= (3,3,3,1,1) \in PRT_{0}[11] $ are given as
		$ARG[(3,3,3,1,1)]= \{ (3,3,3,1,1),(3,3,1,1,3),(3,1,1,3,3), (1,1,3,3,3),\\(1,3,1,3,3), (1,3,3,1,3),(1,3,3,3,1), (3,1,3,1,3),  (3,1,3,3,1), (3,3,1,3,1)\}$.
\end{example}
   \par It is to be the noted that, if $t_1, t_2 , \ldots, t_l$ denote the  $l$ distinct elements in the parts  $b_1,b_2,\ldots,  b_t$ with multiplicity   $r_1,r_2,\ldots,r_l$ respectively so that   $b_1+b_2+\dots+ b_t = \sum\limits_{s=1}^{l} r_s  t_s= r - (i-j) M_w  $, then  $| ARG[b]|= \frac{t!}{r_1 ! \ r_2 ! \ \cdots r_l !}$ for each $b=(b_1,b_2,\ldots,b_t)  \in PRT_{i-j}[r]$.
\par Now, for each $b=(b_1,b_2,\ldots,b_t) \in 	PRT_{i-j}[r]$ and for each $I  \in \mathcal{I}_j^{i}$, we  define a set  $T_b [I] $ which gives all  vectors $x =  x_{1} \oplus x_{2} \oplus \ldots \oplus x_{n} \in \mathbb{F}_q^{N}$ with $(P,w,\pi)$-weight $ r=b_1 + b_2 +  \ldots + b_j +(i-j)M_w  $ such that  $\langle supp_{\pi} (x) \rangle=I $. Thus,
\begin{equation*}
	\begin{split}
		&	T_b [I] =	 \bigcup\limits_{ (b_{m_{i_1}},  b_{m_{i_2}},  \ldots , b_{m_{i_j}})    \in ARG[b] }\\&  \bigg\{    x= x_{1}   \oplus \cdots  \oplus x_{n}    \in \mathbb{F}_{q}^{N} : 
		x_s  =   \left\{ 
		\begin{array}{ll}
			\bar{0}  \in \mathbb{F}_q^{k_s}, & \text{for} ~s \notin I \\
			x_s \in	D_{b_{m_s}}^{k_s}, &   \text{for}  \ s \in Max(I) \\
			x_s	\in \mathbb{F}_q^{k_s}, &  \text{for}  \ s \in I  \setminus Max(I)
		\end{array}	\right\},    1 \leq s \leq n \bigg\}
	\end{split}
\end{equation*}
\begin{proposition}
	For $1 \leq r \leq n M_{w}$, we have   \begin{center}
	$A_r = \bigcup\limits_{i =1}^{n}  \bigcup\limits_{j=1 }^{i}  \bigcup\limits_{I \in \mathcal{I}_j^{i} } \bigcup\limits_{b =(b_1,b_2,\ldots,b_t) \in 	PRT_{i-j}[r] } T_b [I]$.
	\end{center}
\end{proposition}
When $r \leq M_w$,  for any $x \in A_r$, all elements of $I_x^{P,\pi}$ are  maximal (that is, $I_x^{P,\pi}= M_x^{P,\pi}$), so that   $w_{(P,w,\pi)}(x)= \sum\limits_{i \in M_{x}^{P,\pi}}  \tilde{w}^{k_i}{(x_i)}  =  \tilde{w}^{k_{i_1}}(x_{i_1})+  \tilde{w}^{k_{i_2}}(x_{i_2})  + \cdots+  \tilde{w}^{k_{i_j}}(x_{i_j})$ where $ M_x^{P,\pi} = \{i_1,i_2,\ldots,i_j\} $. 
\begin{proposition}\label{rlessthanMw}
	For any $1 \leq r \leq M_{w} $,  the number of $N$-tuples  $x \in \mathbb{F}_{q}^N$ having $w_{(P,w,\pi)}(x)= r$  is $|A_{r}| 
	= $
	\begin{align*}
		\sum\limits_{j=1}^{n} \sum\limits_{I \in  \mathcal{I}_j^{j}} \sum\limits_{(b_1,b_2,\ldots,b_j) \in PRT_0[r]} \sum\limits_{(b_{m_1},  b_{m_2},  \ldots , b_{m_j}) \in ARG[(b_1,b_2,\ldots,b_j)] }  |D_{b_{m_{i_1}}}^{k_{i_1}}| |D_{b_{m_{i_2}}}^{k_{i_2}}| \cdots |D_{b_{m_{i_j}}}^{k_{i_j}}|
	\end{align*}
	wherein, for an $I \in \mathcal{I}_j^{j}$, $Max(I) = \{i_1,i_2,\ldots,i_j\}$ and $(b_{1},b_{2}, \dots, b_{j}) \in PRT_{0}[r]$.
\end{proposition}
\begin{proof} 
	Let $x \in {A}_r $ and $I = \langle supp_{\pi} (x) \rangle $. Now $I  \in \mathcal{I}_j^{j}$ for some $j \in \{1,2,\ldots,n\}$ as $r \leq M_w$. Let $ I = \{i_1,i_2,\ldots,i_j\}$, $b=(b_1,b_2,\ldots,b_j)  \in PRT_{0}[r]$.  Thus, 
	\begin{align*}
 	&	T_b [I] = \bigcup\limits_{ (b_{m_{i_1}},  b_{m_{i_2}},  \ldots , b_{m_{i_j}})    \in ARG[b] } \\& 		
				 \bigg\{  x= x_{1} \oplus x_{2} \oplus \cdots \oplus x_{n} \in \mathbb{F}_{q}^{N} :      x_s  = \left\{ 
			\begin{array}{ll}
				0  \in \mathbb{F}_q^{k_s}, & \text{for} ~s \notin I \\
				x_s \in 	D_{b_{m_s}}^{k_s}, &   \text{for}  \ s \in I
			\end{array}	\right\}, 1 \leq s \leq n \bigg\}
	\end{align*}
	gives all the vectors $x \in \mathbb{F}_q^{N}$ having $w_{(P,w,\pi)}(x)= r$ such that  $ \langle supp_{\pi} (x) \rangle =I $ and $ \sum\limits_{s \in Max(I)}  \tilde{w}^{k_s}(x_s) = b_1 + b_2 +  \cdots + b_j$. So,   
	\begin{equation*}
		|T_b [I] | = \sum\limits_{(b_{m_{i_1}},  b_{m_{i_2}},  \ldots , b_{m_{i_j}}) \in ARG[(b_{1},b_{2}, \dots, b_{j})] } |D_{b_{m_{i_1}}}^{k_{i_1}}|  |D_{b_{m_{i_2}}}^{k_{i_2}}| \cdots |D_{b_{m_{i_j}}}^{k_{i_j}}| 
	\end{equation*}
	Hence, the number of $x \in \mathbb{F}_q^{N}$ having ${(P,w,\pi)}$-weight as $r$ is
	\begin{equation*}	
		|A_{r}| 
		= \sum\limits_{j=1}^{n} \sum\limits_{I \in  \mathcal{I}_j^{j}} \sum\limits_{(b_1,b_2,\ldots,b_j) \in PRT_0[r]} |T_b [I] |
	\end{equation*} 
\end{proof} 
For the case when  $r \geq M_w$, we will have $ r=tM_w+k$ for some $k \leq M_w$ and  for any $x \in A_r$, $I_x^{P, \pi} $ will consist of at most $t$ non-maximal elements.
\begin{proposition}\label{r=Mw}
	For any $ r =t M_{w} + k$ with $1 \leq t \leq n-1$ and $0 < k \leq M_w$, the number of $N$-tuples  $ x \in \mathbb{F}_{q}^N$ having $w_{(P,w,\pi)}(x)= r$  is 
	\begin{align*}
		|A_{tM_w + k}| =\sum\limits_{i=0}^{t} \sum\limits_{j=1}^{n}  \sum\limits_{I \in  \mathcal{I}_j^{j+i}}
		\sum\limits_{(b_1,b_2,\ldots,b_j) \in PRT_{i}[r]} &\sum\limits_{(b_{m_{i_1}},b_{m_{i_2}},  \ldots , b_{m_{i_j}}) \in ARG[(b_1,b_2,\ldots,b_j)] } \\&  |D_{b_{m_{i_1}}}^{k_{i_1}}|   |D_{b_{m_{i_2}}}^{k_{i_2}}|  \cdots |D_{b_{m_{i_j}}}^{k_{i_j}}|   q^{k_{l_1} + k_{l_2} + \cdots + k_{l_i}} 
	\end{align*} 
	wherein, for an $I \in {\mathcal{I}_j^{j+i} }$,  $Max(I)=\{i_1,i_2,\ldots,i_j\}$ and  $I \setminus Max(I) = \{l_1, l_2, \dots, l_{i} \}$.
\end{proposition}
\begin{proof}
	Let $x \in {A}_r $ and $I = \langle supp_{\pi} (x) \rangle $. As  $w_{(P,w,\pi)}(x) =r \geq M_w  $, $I  \in \mathcal{I}_j^{j+i}$ for some $j  \in \{1,2,\ldots,n\}$ and $ i \geq  0$.  
	Let $ Max(I) = \{i_1,i_2,\ldots,i_j\} $, $ I  \setminus Max(I) = \{l_1, l_2, \dots, l_{i} \}$ and $b=(b_1,b_2,\ldots,b_j)  \in PRT_{i}[r]$.  Thus, 
		\begin{align*}
	&	T_b [I] = 	\bigcup\limits_{ (b_{m_{i_1}},   \ldots , b_{m_{i_j}})    \in ARG[b] } 
			\\ & 	 \bigg\{ 
			x= x_{1} \oplus  \cdots \oplus x_{n} \in \mathbb{F}_{q}^{N} :  x_s  = \left\{ 
			\begin{array}{ll}
				\bar{0}  \in \mathbb{F}_q^{k_s}, & \text{for} ~s \notin I \\
				x_s \in	D_{b_{m_s}}^{k_s}, &  \text{for}  \ s \in Max(I) \\
				x_s	\in \mathbb{F}_q^{k_s}, &  \text{for}  \ s \in I  \setminus Max(I)
			\end{array}	\right\},  1 \leq s \leq n  \bigg\}  
	\end{align*}
 gives all the vectors $x \in \mathbb{F}_q^{N}$ having $w_{(P,w,\pi)}(x)= r$ such that  $\langle supp_{\pi} (x) \rangle  = I$ and $ \sum\limits_{i \in Max(I)}  \tilde{w}^{k_i}(x_i) = b_1 + b_2 +  \cdots + b_j$. So,   
	\begin{equation*}
		|T_b [I]| = \sum\limits_{(b_{m_{i_1}},  b_{m_{i_2}},  \ldots , b_{m_{i_j}})) \in ARG[b] } |D_{b_{m_{i_1}}}^{k_{i_1}}|  |D_{b_{m_{i_2}}}^{k_{i_2}}| \cdots |D_{b_{m_{i_j}}}^{k_{i_j}}|  q^{k_{l_1} + k_{l_2} + \cdots + k_{l_i}} 
	\end{equation*}
	As  $0 \leq |I \setminus Max(I)| = i \leq t $, we have 
	$	|A_{r}| 
	=  \sum\limits_{i=0}^{t} \sum\limits_{j=1}^{n} \sum\limits_{I \in  \mathcal{I}_j^{j+i}} \sum\limits_{(b_1,b_2,\ldots,b_j) \in PRT_i[r]} |T_b [I]|$.	
\end{proof}
Based on the results arrived at, $|A_r|$ can be stated  for certain particular values of $r$: \\
When $r < M_w$, we have
  \begin{equation*}
	|A_1| = \sum\limits_{I \in  \mathcal{I}_1^{1}}  |D_{1}^{k_{i_1}} |, ~
	|A_2| = \sum\limits_{I \in  \mathcal{I}_1^{1}} |D_{2}^{k_{i_1}} | + \sum\limits_{I \in  \mathcal{I}_2^{2}}  |D_{1}^{k_{i_1}} | |D_{1}^{k_{i_2}} | ~\text{and so on}.
  \end{equation*}	
   When $r = M_w$, we have 
  \begin{equation*}
	|A_{M_w}| =\sum\limits_{j=1}^{n} \sum\limits_{I \in  \mathcal{I}_j^{j}} 
	\sum\limits_{(b_1,\ldots,b_j) \in PRT_0{[M_w]}}\sum\limits_{(b_{m_{i_1}},   \ldots , b_{m_{i_j}}) \in ARG[(b_1,\ldots,b_j)] } |D_{b_{m_{i_1}}}^{k_{i_1}}|  \cdots |D_{b_{m_{i_j}}}^{k_{i_j}}|  
  \end{equation*}
 When  $r=M_{w}+k$ for  $1\leq k \leq M_{w} -1$, we have 
  \begin{multline*}
	|A_{M_{w} + k}|=\sum\limits_{j=1}^{n} \sum\limits_{I \in  \mathcal{I}_j^{j}} \sum\limits_{(b_1,\ldots,b_j) \in PRT_0[r]}\sum\limits_{(b_{m_{j_1}},    \ldots , b_{m_{j_j}}) \in ARG[(b_1,\ldots,b_j)] } |D_{b_{m_{i_1}}}^{k_{i_1}}|  \cdots |D_{b_{m_{i_j}}}^{k_{i_j}}|   \\ + 
	\sum\limits_{j=1}^{n} \sum\limits_{I \in  \mathcal{I}_j^{j+1}} \sum\limits_{(b_1,\ldots,b_j) \in PRT_{i}[r]}\sum\limits_{(b_{m_{i_1}},  \ldots , b_{m_{i_j}}) \in ARG[(b_1,\ldots,b_j)] } |D_{b_{m_{i_1}}}^{k_{i_1}}|  \cdots |D_{b_{m_{i_j}}}^{k_{i_j}}|    q^{k_{l_1}}		
\end{multline*} 
Now we shall illustrate the obtained results by choosing the field as $\mathbb{Z}_{7}$ by considering the weight $w$ on it as the Lee weight.  Note that, if we let $L_j = \{ i \in [n] : | \langle i \rangle |= j\} $,	 $1 \leq j \leq n$, then  $|\mathcal{I}_j^j| = {|L _1| \choose j}$. If $m_0$ denote the number of minimal elements in the given poset $P$, then  $|L_1|=m_0$.
\begin{example} \label{ex1}
	Let $\preceq $ be a partial order relation  on  $[5]=\{1,2,3,4,5\}$ such that $1 \preceq 2$. Consider $w$ as the Lee weight on $\mathbb{Z}_{7}$ and let $\mathbb{Z}_{7}^{13}=\mathbb{Z}_{7}^2 \oplus \mathbb{Z}_{7}^3 \oplus \mathbb{Z}_{7}^4 \oplus \mathbb{Z}_{7}^2 \oplus \mathbb{Z}_{7}^2$.  Here $I_1 = \{1\}$, $I_2 = \{3\}$, $I_3 = \{4\}$, $I_4 = \{5\}$  are the only ideals with cardinality one with respect to 
	$ \preceq$ and hence $\mathcal{I}_1^1 =\{I_1, I_2, I_3,I_4\}$. $I_5 = \{1,2\}$ is the ideal with cardinality two with one maximal element and $\mathcal{I}_1^{2} =\{I_5\}$.  $I_6 = \{1,3\}$, $I_7 = \{1,4\}$, $I_8 = \{1,5\}$, $I_9 = \{3,4\}$, $I_{10} = \{3,5\}$, $I_{11} = \{4,5\}$ are the ideals with cardinality $2$ with $2$ maximal elements and $\mathcal{I}_2 ^2 =\{I_6, I_7, I_8,I_{9},I_{10},I_{11}\}$.  $I_{12} = \{1,2,3\}$, $I_{13} = \{1,2,4\}$, $I_{14} = \{1,2,5\}$ are the ideals with cardinality $3$ with $2$ maximal elements and $\mathcal{I}_2^{3} =\{I_{12}, I_{13}, I_{14}\}$. $I_{15}= \{1,3,4\}$,  $I_{16}= \{1,3,5\}$, $I_{17}= \{1,4,5\}$, $I_{18}= \{3,4,5\}$ are the ideals with cardinality $3$ with $3$ maximal elements and $\mathcal{I}_3^3 = \{I_{15},I_{16}, I_{17}, I_{18} \}$. $I_{19} = \{1,2,3,4\}$, $I_{20} = \{1,2,3,5\}$, $I_{21} = \{1,2,4,5\}$ are the ideals with cardinality $4$ with $3$ maximal elements and  $\mathcal{I}_3 ^4 =\{I_{19}, I_{20}, I_{21} \}$.  $I_{22} = \{1,3,4,5\}$ is the only ideal with cardinality $4$ with $4$ maximal elements and thus $\mathcal{I}_4 ^4 =\{I_{22} \}$.  $I_{23} = \{1,2,3,4,5\}$ is the only ideal with cardinality $5$ with $4$ maximal elements and hence $\mathcal{I}_4 ^5 =\{I_{23} \}$. And, $\mathcal{I}_1^3 = \mathcal{I}_1^4 = \mathcal{I}_1^5 = \mathcal{I}_2^4 = \mathcal{I}_2^5 =  \mathcal{I}_3^5 =  \mathcal{I}_5^5 = \{ \}$. 
	\begin{enumerate}[label={(\roman*)}]
		\item   \sloppy{In $\mathbb{Z}_7$, $M_w=3$ and $|D_{1}| = |D_{2}| =|D_{3}| =2 $. 
		From Proposition \ref{D_r^j}, in $\mathbb{Z}_7^{13}$,} \\ 
		$|D_{1}^{k}| = (1+|D_{1}| )^{k} -1$, \\
		$|D_{2}^{k} | = (1+|D_{1}|+|D_{2}| )^{k} -(1+|D_{1}|)^k$, \\
		$|D_{3}^{k} | = (1+|D_{1}|+|D_{2}| +|D_{3}|)^{k} -(1+|D_{1}|+|D_{2}|)^k$, 
		where $k =2,3$, $4$.
		\item 	    \sloppy{$PRT[3]=\{(1 ,1 ,1),  ( 1,2), (3)\}$ and   $ARG[( 1,2)] = \{ ( 1,2), ( 2,1) \}$. 
		Thus,	number of vectors $x \in  \mathbb{Z}_{7}^{13}$ having weight $3$ is }
		\begin{equation*}
			\begin{split}
		|A_3|&= 	 \sum\limits_{j=1}^{3} \sum\limits_{I \in  \mathcal{I}_j^{j}} \sum\limits_{(b_1,\ldots,b_j) \in PRT_{i}[r]} \sum\limits_{(b_{m_{i_1}},    \ldots , b_{m_{i_j}}) \in ARG[(b_1,\ldots,b_j)] } |D_{b_{m_{i_1}}}^{k_{i_1}}|  |D_{b_{m_{i_2}}}^{k_{i_2}}| \cdots |D_{b_{m_{i_j}}}^{k_{i_j}}|  \\ 
				&=   \sum\limits_{I \in  \mathcal{I}_1^{1}} |D_{3}^{k_{i_1}}| + \sum\limits_{I \in  \mathcal{I}_2^{2}} (|D_{1}^{k_{i_1}}|  |D_{2}^{k_{i_2}}| + |D_{2}^{k_{i_1}}|  |D_{1}^{k_{i_2}}|)  + \sum\limits_{I \in  \mathcal{I}_{3}^{3}} |D_{1}^{k_{i_1}}|  |D_{1}^{k_{i_2}}|  |D_{1}^{k_{i_3}}| \\ 	&	=  	|D_{3}^{2} | + |D_{3}^{4} | + |D_{3}^{2} |  +|D_{3}^{2} | 
				+  |D_{1}^{2} | |D_{2}^{4} | + |D_{2}^{2} | |D_{1}^{4} | + |D_{1}^{2} | |D_{2}^{2} | + |D_{2}^{2} ||D_{1}^{2} | +  \\
					&  \ \ \ \ |D_{1}^{2} |  |D_{2}^{2} | +  |D_{2}^{2} | |D_{1}^{2} | 
				+    |D_{1}^{4} | |D_{2}^{2} |+ |D_{2}^{4} | |D_{1}^{2} | +  |D_{1}^{4} | |D_{2}^{2} | + |D_{2}^{4} | |D_{1}^{2} | + |D_{1}^{2} | |D_{2}^{2} |  \\ 
				&   \ \ \    + |D_{2}^{2} | |D_{1}^{2} |  +    |D_{1}^{2} |   |D_{1}^{4} |  |D_{1}^{2} | +    |D_{1}^{2} |   |D_{1}^{4} |   |D_{1}^{2} | 	+ 	 |D_{1}^{2} |   |D_{1}^{2} |  |D_{1}^{2} | 	 +  |D_{1}^{4} |  |D_{1}^{2} |  |D_{1}^{2} | \\		
						&  =24 + 1776 + 24 + 24 +  8 \times 544 +  16 \times 80 +  8 \times 16 +  16 \times 8 +   8 \times 16 +  16 \times 8 + \\ &  \ \ \ \  80 \times 16 +  544 \times 8  +  80 \times 16 +  544 \times 8 +  8 \times 16 + 16 \times 8 +
					8 \times 80  \times 8 +				 \\		
					 & 	\ \ \ \ 8 \times 80  \times 8	+  8 \times 8  \times 8 + 80 \times 8  \times 8 			   
						\end{split}
				\end{equation*}
						\begin{equation*}									
					|A_3|	= 35,384.				
		\end{equation*}
		\item      \sloppy{Also,  $PRT[14]=\{(3,3,3,3,2)\}$ and $PRT_1[14]=\{(3,3,3,2)\}$.   Consider the arrangements of $ (3,3,3,2)$ as $ARG[(3,3,3,2)] = \{ (3,3,3,2),(3,3,2,3),(3,2,3,3),(2,3,3,3)\}$. Now  $\mathcal{I}_4 ^5 =\{I_{23} \}$ and $\mathcal{I}_5^5 = \{\phi\}$.}  Thus,	number of vectors $x \in  \mathbb{Z}_{7}^{13}$ having weight $14$ is  
		\begin{equation*}
			\begin{split}			
				|A_{14}| &= 		\sum\limits_{I \in  \mathcal{I}_{4}^{5}} 			  \sum\limits_{(b_1,b_2,b_3,b_4) \in PRT_{1}[14]} \sum\limits_{(b_{m_{i_1}},  b_{m_{i_2}},  b_{m_{i_3}} , b_{m_{i_4}}) \in ARG[(3,3,3,2)] }   |D_{b_{m_{k_{i_1}}}}^{k_{i_1}}|  |D_{b_{m_{i_2}}}^{k_{i_2}}| \\& {\hspace{9cm}     |D_{b_{m_{i_3}}}^{k_{i_3}}| |D_{b_{m_{i_4}}}^{k_{i_4}}|  q^{k_{l_1}}} \\
		  &=(|D_{3}^{3}|  |D_{3}^{4}| |D_{3}^{2}|  |D_{2}^{2}|   +  |D_{3}^{3}| |D_{3}^{4}|  |D_{2}^{2}|  |D_{3}^{2}| + |D_{3}^{3}|  |D_{2}^{4}|  |D_{3}^{2}|  |D_{3}^{2}|  +\\& \ \  \  \ \ |D_{2}^{3}|  |D_{3}^{4}|   |D_{3}^{2}| |D_{3}^{2}|)  7^2\\ 
				&= (218 \times 1776 \times 24 \times 16 + 218 \times 1776 \times 16 \times 24 + 218 \times 544 \times 24 \times 24 \\& \ \ \ \  + 98 \times 1776 \times 24 \times 24 ) \times 49 \\
				&= 22,829,377,536.
			\end{split}
		\end{equation*}			
	\end{enumerate}
\end{example}
Calculating $ARG[b]$ for each  partition $ b= (b_1,b_2,\ldots,b_j)  \in 	PRT_{i-j}[r] $ is seemingly cumbersome. But, that is not the case  when   $k_i = k$ $\forall$ $i \in [n]$, as the  number of components $x_i \in \mathbb{F}_{q}^{k_i} $ having $\tilde{w}^{k_i}(x_i)=b_i$ will not then depend on the label $i \in [n]$ and thus it is not  necessary to  find  $ARG[b]$. If $ b= (b_1,b_2,\ldots,b_j)  \in 	PRT_{i-j}[r] $, let $t_1, t_2 , \ldots, t_l$ be the  distinct  $l$ elements from  the $j$ parts  $b_1,b_2,\ldots,  b_j$ with multiplicity   $r_1,r_2,\ldots,r_l$ respectively so that   $ \sum\limits_{s=1}^{l} r_s  t_s=b_1+b_2+\dots+ b_j $.
\begin{proposition}\label{Ar without ARG}
	Let $k_i=k$ $\forall$ $i \in [n]$. Then, for any $ r \leq n M_w$, the number of $N$-tuples  $ x \in \mathbb{F}_{q}^N$ having $w_{(P,w,\pi)}(x)= r$  is $|A_{r}| =$
	\begin{align*}
		\sum\limits_{i=1}^{n} \sum\limits_{j=1}^{i} \sum\limits_{I \in \mathcal{I}_j^{i}}  \sum\limits_{(b_1,b_2,\ldots,b_j) \in PRT_{i-j}[r]} q^{k(i-j)}    \prod\limits_{s=1}^{l} |D_{t_s}^k | ^{r_s}     {j - (r_1+r_2+\ldots+r_{s-1})\choose r_s} 
	\end{align*}
	where  $r_s$ parts among the $j$ parts in  $b_1, b_2, \dots, b_j  $  are equal to $t_s$ 
	 $(1 \leq s \leq l)$.
\end{proposition}
\begin{proof}
	Let $I \in \mathcal{I}_{j}^{i}$ with  $|Max(I)|=j$ and let $ b=(b_1,b_2,\ldots,b_j)  \in PRT _{i-j}[r] $.  Let $t_1, t_2 , \ldots, t_l$ be the distinct $l$  elements from the set of $j$ parts in $b$ such that $r_s$ is the multiplicity of  $t_s$ ($1 \leq s \leq l$) so that $\sum\limits_{s=1}^{l}  r_s  = j $ and $\sum\limits_{s=1}^{l} r_s t_s  = b_1+b_2+\cdots +b_j =r - (i-j) M_w $.
	Since $t_s$  occurs $r_s$ times, the total number of choices for $x_s \in \mathbb{F}_{q}^{k_s} $ such that $\tilde{w}^{k_s} (x_s)=t_s$ in $r_s$ places of $Max(I)$ is $|D_{t_s}^k | ^{r_s} {j \choose r_s}$. Then for each $ (b_1,b_2,\ldots,b_j)  \in PRT _{i-j}[r] $ and $I \in \mathcal{I}_{j}^{i}$, the total number of $N$-tuples $x \in \mathbb{F}_q ^N$ having $w_{(P,w,\pi)}(x)= r $ is $ q^{k(i-j)} \prod\limits_{s=1}^{l}|D_{t_s}^k | ^{r_s} {j - (r_1+r_2+\cdots+r_{s-1})\choose r_s} $. Hence, $	|A_{r}| =$
	\begin{align*}
		\sum\limits_{i=1}^{n} \sum\limits_{j=1}^{i} \sum\limits_{I \in \mathcal{I}_j^{i}} \sum\limits_{ (b_1,b_2,\ldots,b_j)  \in PRT _{i-j} [r] }  q^{k(i-j)} \prod\limits_{s=1}^{l}
		|D_{t_s}^k | ^{r_s}  {j - (r_1+r_2+\ldots+r_{s-1})\choose r_s} 
	\end{align*} where  $r_0 = 0$ and $|D_{t_s}^k |= ( \sum\limits_{i=0}^{t_s} |D_{i}|)^{k}-( \sum\limits_{i=0}^{t_s-1} |D_{i}|)^k$.
\end{proof}
\begin{corollary}
	If $k_i=k$ $\forall$ $i \in [n]$ then the number of $N$-tuples  $ x \in \mathbb{F}_{q}^N$ having $w_{(P,w,\pi)}(x)= n M_w$  is 
$
		({ q^{k}-( q -  |D_{M_{w}}|) ^{k} })^{t} q^{k(n-t)}  
	$ 
	where $t$ is the  number of maximal elements in the given poset $P$. 
\end{corollary}
As the $d_{(P,w,\pi)}$-metric is translation invariant, $B_{r}(x) = x + B_{r}(0)$ for an $ x \in \mathbb{F}_{q}^N $ where $ 0 \in \mathbb{F}_{q}^N $.  Moreover,  $| B_{r}(0) | = 1 + \sum\limits_{t=1}^{r} 	|S_t (0)|$ and $ |S_t (0)| =|A_t| $. Thus,
\begin{proposition} \label{wpirball}
	The number of $N$-tuples in a  $(P,w,\pi)$-ball of radius $r$ centered at $ x \in \mathbb{F}_{q}^N $ is  $| B_{r}(x) | = 1 + \sum\limits_{t=1}^{r} 	|A_t|$. 
\end{proposition} 
In the remaining part of this section, we show how  $|A_r| $ can be arrived at for various spaces viz., $(P,w)$-space, $(P,\pi)$-space, $\pi$-space and $P$-space.
\subsection{$(P,w)$-space}
The  $(P,w,\pi)$-space becomes $(P,w)$-space \cite{wcps} by considering   $k_i =1$ for all $i \in [n]$ (see Remark \ref{becomes}). Here $N=\sum\limits_{i=1}^{n} k_i =n$ . For $u =( u_1 , u_2 , \ldots ,u_n ) \in \mathbb{F}_{q}^{n} $,  $ supp_{\pi}( u)=\{i \in [n] : u_i \neq 0\} = supp(u)$ and $ \tilde{w}^{k_i}{(u_i)} =  w{(u_i)} ~ \forall ~i \in [n]$. Thus,    
\begin{equation*}
	w_{(P,w)}(u)= \sum\limits_{i \in M_{u}^P}  w{(u_i)} + \sum\limits_{i \in {I_u}^{P} \setminus M_{u}^P} M_w.  
\end{equation*}
\par 
   Now, $A_r = \{u= ( u_1 , u_2 , \ldots ,u_n ) \in \mathbb{F}_{q}^{n}  : w_{(P,w)}(u)=r  \}$
   and by Proposition \ref{rlessthanMw}, for any $1 \leq r \leq M_{w} -1$, the number of $N$-tuples  in $ \mathbb{F}_{q}^n$ having $w_{(P,w)}(x)= r$  is
\begin{align*}
	|A_{r}| 
	&= \sum\limits_{j=1}^{n} \sum\limits_{I \in  \mathcal{I}_j^{j}} \sum\limits_{(b_1,b_2,\ldots,b_j) \in PRT_0[r]} \sum\limits_{(b_{m_{i_1}},  b_{m_{i_2}},  \ldots , b_{m_{i_j}} ) \in ARG[b] } |D_{b_{m_{i_1}}}| |D_{b_{m_{i_2}}}| \cdots |D_{b_{m_{i_j}}}| \\
	&= \sum\limits_{j=1}^{n} \sum\limits_{I \in  \mathcal{I}_j^{j}} \sum\limits_{(b_1,b_2,\ldots,b_j) \in PRT_0[r]}  |D_{b_{m_{i_1}}}|  |D_{b_{m_{i_2}}}| \cdots |D_{b_{m_{i_j}}}| |ARG[(b_1,b_2,\ldots,b_j)]|
\end{align*}
wherein, for an $I \in \mathcal{I}_{j}^{j}$, $\{i_1,i_2,\ldots,i_j\}$ is the set of maximal elements of $I$.
\par By Proposition \ref{r=Mw}, for $ r =t M_{w} + k$; $1 \leq t \leq n-1$, $0 \leq k \leq M_w$, we can get, $|A_{tM_w + k}|=$
\begin{align*}
& 		\sum\limits_{i=0}^{t} \sum\limits_{j=1}^{n} \sum\limits_{I \in  \mathcal{I}_{j}^{i+j}} \sum\limits_{(b_1,b_2\ldots,b_j) \in PRT_{i}[r]} \sum\limits_{(b_{m_{i_1}}, b_{m_{i_2}},   \ldots , b_{m_{i_j}} ) \in ARG[b] } |D_{b_{m_{i_1}}}|  |D_{b_{m_{i_2}}}| \cdots |D_{b_{m_{i_j}}}|  q^{i} \\	&=\sum\limits_{i=0}^{t} \sum\limits_{j=1}^{n} \sum\limits_{I \in  \mathcal{I}_{j}^{i+j}} \sum\limits_{(b_1,b_2,\ldots,b_j) \in PRT_{i}[r]} |D_{b_{m_{i_1}}}|  |D_{b_{m_{i_2}}}| \cdots |D_{b_{m_{i_j}}}|   q^{i} |ARG[(b_1,b_2,\ldots,b_j)]|
\end{align*}
 \par Alternatively,  from Proposition \ref{Ar without ARG},  we can determine  the cardinality of $A_r$ without finding $ARG[b]$, as $k_i=1 ~\forall ~ i \in [n]$. 
\begin{corollary} \label{Ar for pw dist without ARG}
	For any $1 \leq r \leq n M_{w} $, the number of $n$-tuples  $u \in \mathbb{F}_{q}^n$ having $w_{(P,w)}(u)= r$  is
	\begin{align*}
		|A_{r}| = 	\sum\limits_{i=1}^{n} \sum\limits_{j=1}^{i} \sum\limits_{I \in \mathcal{I}_j^{i}}  \sum\limits_{ (b_1,b_2,\ldots,b_j)  \in PRT _{i-j} [r] } q^{i-j}  \prod\limits_{s=1}^{l}  {|D_{t_s}|}^{r_s} {j - (r_1+r_2+\ldots+r_{s-1})\choose r_s} 
	\end{align*}
	where  $r_s$ parts among the $j$ parts in  $b_1, b_2, \dots, b_j  $  are equal to $t_s$  ($1 \leq s \leq l$).
\end{corollary}
\begin{corollary}
	The number of $n$-tuples  $ x \in \mathbb{F}_{q}^n$ having the $(P,w,\pi) $-weight $nM_w$  is 	$ {|D_{M_w}|}^{t} q^{n-t}  $ where $t$ is the  number of maximal elements of the given poset $P$. 
\end{corollary}
If $P=([n],\preceq)$ is a chain (here $1 \preceq 2 \preceq \ldots \preceq n$), each ideal $I \in \mathcal{I}(P)$  will have only one maximal element.  Let  $v \in \mathbb{F}_q^n$,  $I_v^P=  \langle supp (v) \rangle $ and $j$ be the maximal element of $I_v^P$. Then, 
$	w_{(P,w)}(v)= w{(v_j)} + (|I_v^P|-1) M_w $.   Moreover, 
\begin{itemize}
	\item   For $1 \leq r \leq M_{w} $, we have $v \in A_r$ iff $| \langle supp (v) \rangle | = 1 $, and hence $	|A_{r}| = |D_{r}|$.
	\item	For $r= tM_{w}+ k $, where $0 < k \leq M_{w}$ and $1 \leq t \leq n-1$, we have $v \in A_r$ iff $| \langle supp (v) \rangle | = t + 1 $. Thus,
	$|A_{tM_{w} +  k }| =  q^{t} |D_{k}|$.
\end{itemize} 
\subsection{$(P,\pi)$-space}
 Given a poset $P=([n],\preceq)$, a label map $\pi$ on $[n]$, and considering  $w$ to be the Hamming weight on  $ \mathbb{F}_{q}$,  the $(P,w,\pi)$-space becomes poset block or  $(P,\pi)$-space \cite{Ebc} (see Remark \ref{becomes}). Now $M_w= 1$, and hence we have
   $|D_0| = 1$, $|D_1| =  q-1$ and $|D_1^{k_j}| = q ^{k_j}- 1$ for $1 \leq j \leq n$.  
Clearly,  $x$ $\in A_r$ iff $ | \langle supp_{\pi} (x) \rangle | =r $.  
  \par   As $M_w= 1$, we have $|PRT_{i-j}[r]|=1$  for any $r \leq n$, for      $PRT_{i-j}[r]=\{(1,1,\ldots,1) : 1+1+\cdots+1=r-(i-j)\}$.   
 From Proposition  \ref{r=Mw},   the number of $N$-tuples  $ x \in \mathbb{F}_{q}^N$ having $w_{(P,\pi)}(x)= r$  is
\begin{align*}
	\sum\limits_{i=0}^{r-1} \sum\limits_{j=1}^{n} \sum\limits_{I \in  \mathcal{I}_j^{j+i}} \sum\limits_{(1,1,\ldots,1) \in PRT_{i}[r]}   |D_{1}^{k_{i_1}}|  |D_{1}^{k_{i_2}}| \cdots |D_{1}^{k_{i_j}}| q^{k_{l_1} + k_{l_2} + \cdots + k_{l_i}} 
\end{align*}
Since $x$ $\in A_r$ iff $  I_x^P \in \mathcal{I}_j^{r} $ for some $ j \leq r$, we have
\begin{equation*}
	|A_{r}| =   \sum\limits_{j=1}^{r} \sum\limits_{I \in  \mathcal{I}_j^{r}}  (q^{k_{i_1}}-1) (q^{k_{i_2}}-1)  \cdots (q^{k_{i_j}}-1)    q^{k_{l_1} + k_{l_2} + \cdots + k_{l_{r-j}}} 
\end{equation*}
where $ Max(I) = \{i_1,i_2,\ldots,i_j\} $ and $ I  \setminus Max(I) = \{l_1, l_2, \dots, l_{r-j} \}$ for a given  $I  \in \mathcal{I}_j^{r}$.
Moreover, if $k_i = k$ for all $i \in [n]$, then $|A_{r}| =   \sum\limits_{j=1}^{r} \sum\limits_{I \in  \mathcal{I}_j^{r}}  (q^{k}-1)^{j}   q^{k(r-j)} 
$.
\begin{corollary}
	Let	$k_i = k$ for all $i \in [n]$.	Then, $x \in A_n$ iff  $I_x^{P,\pi}  \in  \mathcal{I}_t^{n} $ where $t$ is the  number of maximal elements of the given poset $P$. Moreover,
	$ | A_n | =  (q^{k}-1)^{t} q^{k(n-t)}  $. 
\end{corollary}
If  $P=([n],\preceq)$ is  a chain, there is only one  maximal element in any ideal $I \in \mathcal{I}(P)$ and $|\mathcal{I}_1^{r}|=1$ for $1 \leq r \leq n $. Let  $v \in \mathbb{F}_q^n$ and  $I_v^{P,\pi}= \langle supp_{\pi}(v) \rangle $ with maximal element $j$. Then, 
\begin{align*}
	w_{(P,w,\pi)}(v)= \tilde{w}^{k_j}{(v_j)} + (|I_v^{P,\pi}|-1) M_w   = |I_v^{P,\pi}|= |\langle j \rangle| = w_{(P,\pi)}(v)
\end{align*}
This weight coincides with the RT-weight if $k_i=1$  for each $i \in [n]$. 	  
We have $x$ $\in A_r$ iff $ | \langle supp_{\pi} (x) \rangle | =r $, and thus $
|A_{r}| = (q^{k_r}-1)  q^{k_1+k_2+\dots+k_{r-1}}$ for $1 \leq r \leq n $. Moreover, if $k_i = k$ for each $i \in [n]$, then $ |A_{r}| = (q^{k}-1)  q^{k({r-1})} $ for $1 \leq r \leq n $.
\subsection{ $\pi$-space}
The $(P,w,\pi)$-space becomes the classical $( \mathbb{F}_q^N,d_{\pi})$-space \cite{fxh} by considering  $P=([n],\preceq)$ as  an anti-chain and $w$ as the Hamming weight on $\mathbb{F}_q$ (see Remark \ref{becomes}). Now $A_r = \{x= x_{1} \oplus x_{2} \oplus \ldots \oplus x_{s} \in \mathbb{F}_{q}^{N}  : w_{\pi}(x)=r  \}$ for each $1 \leq r \leq n$. Clearly,  $x$ $\in A_r$ iff $ | supp_{\pi} (x) | =r $.  
\par Now each subset of $[n] $ is an ideal. Let $P_r$ be the collection of all subsets $B$ of $[n]$ such that $|B|=r$. If $B=\{i_1,i_2,\ldots,i_r\} \subseteq [n]$ then  $ | A_r | =  \sum\limits_{  B \in P_r} (q^{k_{i_1}}-1)(q^{k_{i_2}}-1)\cdots(q^{k_{i_r}}-1) $. 	If $k_i=k$ for each $i \in [n]$, $ | A_r | =  {n \choose r} (q^{k}-1)^r $ for $1 \leq r \leq n $.  
\subsection{ $P$-space}
 The $(P,w,\pi)$-space becomes the classical $(\mathbb{F}_{q}^{n}, d_{P})$-poset space \cite{Bru}  by considering $w$ 
 as the Hamming weight on $\mathbb{F}_q$ and $k_i=1$ $\forall$ $i \in [n]$ (see Remark \ref{becomes}). 
 As $M_w=1$,  $\tilde{w}^{k_s}(x_s)=w(x_s)=1$  for all $x_s \in \mathbb{F}_q^{k_s}\setminus \{0\}$, and  $|PRT_{i-j}[r]|=1$  for any $r \leq n$  so that $ARG[b]=PRT_{i-j}[r]$.   For  each $ 1 \leq r \leq n$, if  $A_r = \{x= (x_{1},  x_{2}, \ldots,  x_{n}) \in \mathbb{F}_{q}^{n}  : w_{P}(x)=r  \}$, then  $x$ $\in A_r$ iff $ | \langle supp (x) \rangle | =r $. Therefore, $	|A_{r}| = \sum\limits_{j=1}^{r} \sum\limits_{I \in  \mathcal{I}_j^{r}}  ( q-1 )^{j}  q^{r-j}  $. In particular,   $ A_n= (q-1)^{t} q^{n-t}$ where $t$ is the number of maximal elements in  $P$. If  $P=([n],\preceq)$ is a chain, then $|A_{r}| ={q}^{r-1} {(q-1)}   $ for $ 1 \leq r \leq n$.   
\section{Singleton Bound for $(P,w,\pi)$-code }
In line with the  Singleton bounds for the poset codes \cite{hkmdspc}, the  $(P,\pi)$-codes  \cite{bkdnsr} and  the  pomset codes \cite{gr}, we extend the Singleton bound  for $(P,w,\pi)$-block codes in this section. Let $d_{(P,w,\pi)}\mathbb{(C)}$ be the minimum distance of a block code $\mathbb{C}$ of length $N$ over $\mathbb{F}_q$. Recall that $\mathcal{I}^{i}$ denote the collection of all ideals in $\mathcal{I}(P)$ whose cardinality is $i$.
\begin{theorem}[Singleton bound for $(P,w,\pi)$-block code] \label{sbwpic}
	Let $\mathbb{C} $  be a $(P,w,\pi)$-code of length $N = k_1 + k_2 + \ldots + k_n$ over $\mathbb{F}_q$ with minimum distance $d_{(P,w,\pi)}(\mathbb{C})$ and let $r = \big\lfloor \frac{d_{(P,w,\pi)} \mathbb{(C)}-1}{M_w} \big\rfloor $. Then    $ \max\limits_{J \in  \mathcal{I}^{r}}    \big\{\sum_{i \in J} k_{i}\big\} \leq N - \lceil log_{q}|\mathbb{C}| \rceil $.  
\end{theorem}
\begin{proof}
	\sloppy{There exist two distinct codewords $c_1, c_2 \in \mathbb{F}_q^N $ such that $d_{(P,w,\pi)} (\mathbb{C}) = d_{(P,w,\pi)} {(c_1,c_2)} $. Let $I = \langle supp_{\pi} (c_1-c_2) \rangle $. If one denotes $c_1=c_{11} \oplus c_{12} \oplus \cdots \oplus  c_{1n} $, $c_2=c_{21} \oplus c_{22} \oplus \cdots \oplus  c_{2n} $, with $c_{ij}  \in \mathbb{F}_q^{k_j}$ for $i=1,2$, then $d_{(P,w,\pi)} {(c_1,c_2)} = \sum\limits_{i \in Max(I)}  \tilde{w}^{k_i}(c_{1i}-c_{2i}) + \sum\limits_{i \in {I \setminus Max(I)}} M_w$. Now, $d_{(P,w, \pi)}(\mathbb{C}) - 1 < d_{(P,w, \pi)}(\mathbb{C}) \leq M_{w} |I|$.  Thus,}    $\big\lfloor \frac{d_{(P,w,\pi)} \mathbb{(C)}-1}{ M_w} \big\rfloor < |I|$. Let $r = \big\lfloor \frac{d_{(P,w,\pi)} \mathbb{(C)}-1}{ M_w}\big\rfloor$. There always exists an ideal $J$ of cardinality $r$ such that $J \subset I$ (see \cite{hkmdspc}).  For any ideal $J \in \mathcal{I}^{r}$, every two distinct codewords of $\mathbb{C}$ must be different in some $i^{th}$ block in $[n]$ for an  $i \not \in J$. Otherwise 
	$ d_{(P,w, \pi)} {(c_1,c_2)} = \sum\limits_{i \in Max(J) }  \tilde{w}^{k_i}(c_{1i}-c_{2i}) + \sum\limits_{i \in {J}  \setminus  Max(J)}	M_w \leq |J| M_w \leq d_{(P,w, \pi)}(\mathbb{C}) - 1$, which is a contradiction. Thus, every two distinct codewords of $\mathbb{C}$ must be different for some $i$-label in $[n]$ for an  $i \not \in J$. This proves that there is an injective map from $\mathbb{C}$ to $\mathbb{F}_q^{N - \sum_{i \in J} k_{i} }$. Thus, $|\mathbb{C}| \leq q^{N -\sum_{i \in J} k_{i} }$ and  hence,  
	$ { \sum_{i \in J} k_{i} } \leq  N - \lceil \log_q |\mathbb{C}| \rceil$. 
	As this is true for any ideal $J \in \mathcal{I}^{r}$, we have  
	$ 	\max\limits_{J \in  \mathcal{I}^{r}}    \big\{\sum_{i \in J} k_{i}\big\} \leq N - \lceil log_{q}|\mathbb{C}| \rceil$ where $r = \big\lfloor \frac{d_{(P,w,\pi)} \mathbb{(C)}-1}{M_w} \big\rfloor $.
\end{proof} 
If $k_i = 1 $ $\forall$  $i \in  [n]$  then  $	\max\limits_{J \in  \mathcal{I}^{r}}  \big\{\sum_{i \in J} k_{i}  \big \}=r$. Thus, we get the Singleton bound for a $(P,w)$-code as:
\begin{corollary}[Singleton bound for $(P,w)$-code]
	Let $\mathbb{C} $  be a $(P,w)$-code of length $n$ over $\mathbb{F}_q$ with minimum distance $d_{(P,w)}(\mathbb{C})$.  Then 
	$  \big\lfloor \frac{d_{(P,w)} \mathbb{(C)}-1}{M_w} \big\rfloor   \leq n - \lceil log_{q}|\mathbb{C}| \rceil $. 
\end{corollary}
\begin{remark}
 The Singleton bound  for a $(P,\pi)$-linear code $ \mathbb{C} $ of length $N$ takes the form to be $	\max\limits_{J \in  \mathcal{I}^{r}}    \big\{\sum_{i \in J} k_{i}\big\} \leq N - \lceil log_{q}|\mathbb{C}| \rceil $ as described in \cite{bkdnsr} where 	$d_{(P,\pi)}\mathbb{(C)}$ is the minimum distance of $\mathbb{C}$ and $r = d_{(P,\pi)} \mathbb{(C)}-1$.
\end{remark}
\begin{remark}
	If $\mathbb{C} \subseteq \mathbb{F}_q^n$  is a $P$-code of length $n$ over $\mathbb{F}_q$ with minimum distance $d_{P}(\mathbb{C})$, then  $r =d_{P} \mathbb{(C)}-1$ and $\max\limits_{J \in  \mathcal{I}^{r}}    \big\{\sum_{i \in J} k_{i}\big\} = r$ as $M_w=1$. Thus  the Singleton bound  for $P$-code is $d_{P} \mathbb{(C)}-1   \leq n - \lceil log_{q}|\mathbb{C}| \rceil$ as described in \cite{hkmdspc}. 
\end{remark}
\begin{definition} \label{d1}
	A $(P,w,\pi)$-code $\mathbb{C}  $ of length $N$ over $\mathbb{F}_q$ is said to be a maximum distance separable (MDS) $(P,w,\pi)$-code if it attains its Singleton bound.
\end{definition}
	For the case $k_i = k $  $\forall$ $i \in [n]$, we give a necessary condition for a code to be MDS. 
\begin{theorem}
	Let $k_i = k $  $\forall$ $i \in [n]$ and $\mathbb{C} $  be a $(P,w,\pi)$-code of length $N$  over  $\mathbb{F}_q$ with minimum distance $d_{(P,w,\pi)}(\mathbb{C})$.  If $\mathbb{C}$ is MDS, then $
	{M_w}(n - \frac{\lceil log_{q}|\mathbb{C}| \rceil}{k})    + 1  \leq  d_{(P,w,\pi)} \mathbb{(C)} \leq   {M_w}(n - \frac{\lceil log_{q}|\mathbb{C}| \rceil}{k}+1) $. 
\end{theorem}
\begin{proof}
	If $\mathbb{C}$ is  MDS, there exist an ideal $J$ in $\mathcal{I}^{r}$ such that  $
	\sum\limits_{i \in J} k_{i}  = N - \lceil log_{q}|\mathbb{C}| \rceil $. Since
	 $k_i =k $  $\forall$ $i \in [n]$, so $	r k  = n k - \lceil log_{q}|\mathbb{C}| \rceil $. As
	$r = \big\lfloor \frac{d_{(P,w,\pi)} \mathbb{(C)}-1}{M_w} \big\rfloor $, we have 
	$	n - \frac{\lceil log_{q}|\mathbb{C}| \rceil}{k} \leq  \frac{d_{(P,w,\pi)} \mathbb{(C)}-1}{M_w} < n - \frac{\lceil log_{q}|\mathbb{C}| \rceil}{k} +1 $. Hence, $
	{M_w}(n - \frac{\lceil log_{q}|\mathbb{C}| \rceil}{k})    + 1 \leq  d_{(P,w,\pi)} \mathbb{(C)} \leq   {M_w}(n - \frac{\lceil log_{q}|\mathbb{C}| \rceil}{k}+1)$.		
\end{proof}
    Thus, if $k_i = k $  $\forall$ $i \in [n]$ and  $\mathbb{C} $ is a $(P,w,\pi)$-code of length $N$ over $\mathbb{F}_q$ with minimum distance $d_{(P,w,\pi)}(\mathbb{C})$,  then $\mathbb{C}$ cannot be an  MDS  poset code whenever
    $	1 \leq d_{(P,w,\pi)} \mathbb{(C)} \leq 	{M_w}(n - \frac{\lceil log_{q}|\mathbb{C}| \rceil}{k})  $ or $ d_{(P,w,\pi)} (\mathbb{C})  > {M_w}(n - \frac{\lceil log_{q}|\mathbb{C}| \rceil}{k}+1) $.
  \par Now we will compare the maximum distance separability of codes with different poset metric structures. Let $d_{(P,w,\pi)}(\mathbb{C})$  and  $d_{(P,\pi)}(\mathbb{C})$ be the minimum distances of a code $\mathbb{C} $ of length $N = k_1 + k_2 + \ldots + k_n$ over $\mathbb{F}_q$, with respect to $(P,w,\pi)$-metric and $(P,\pi)$-metric respectively. Clearly, $ d_{(P,w,\pi)} \mathbb{(C)} \leq M_w d_{(P,\pi)} \mathbb{(C)}$. Thus, we have the following:
\begin{proposition} \label{phg}
	Let $d_{(P,w,\pi)}(\mathbb{C})$  and  $d_{(P,\pi)}(\mathbb{C})$ be the minimum distances of a code $\mathbb{C} $ of length $N = k_1 + k_2 + \ldots + k_n$ over $\mathbb{F}_q$, with respect to $(P,w,\pi)$-metric and $(P,\pi)$-metric respectively.  Then $	\big\lfloor \frac{d_{(P,w,\pi)} \mathbb{(C)}-1}{ M_w}\big\rfloor \leq d_{(P,\pi)}(\mathbb{C}) - 1$.	
\end{proposition}
\begin{proposition} \label{pwtoP singleton}
	Let $d_{(P,w)}(\mathbb{C})$  and  $d_{P}(\mathbb{C})$ be the minimum distances of a code $\mathbb{C} $ of length $n$ over $\mathbb{F}_q$, with respect to $(P,w)$-metric and $P$-metric respectively.  Then $	\big\lfloor \frac{d_{(P,w)} \mathbb{(C)}-1}{ M_w}\big\rfloor \leq d_{P}(\mathbb{C}) - 1$.	
\end{proposition}
  We answer affirmatively the question whether a code which is MDS in a weighted  poset (block) metric is also MDS in  poset (block) metric in the following successive results. Recall that,  the   Singleton bound \cite{bkdnsr} of any  $(P,\pi)$-code $\mathbb{C} $  is $ \max\limits_{J \in  \mathcal{I}^{d_{(P,\pi)}(\mathbb{C}) - 1}}    \big\{\sum_{i \in J} k_{i}\big\} \leq  N - \lceil log_{q}|\mathbb{C}| \rceil $.
\begin{theorem}
	If $\mathbb{C} $  is an MDS code with respect to $(P,w,\pi)$-metric then $\mathbb{C} $  is  MDS with respect to $(P,\pi)$-metric.
	\begin{proof}
		\sloppy{If $\mathbb{C} $  is  MDS  with respect to $(P,w,\pi)$-metric then $ \max\limits_{J \in  \mathcal{I}^{r}}    \big\{\sum_{i \in J} k_{i}\big\} = N - \lceil log_{q}|\mathbb{C}| \rceil $ where $r = \big\lfloor \frac{d_{(P,w,\pi)} \mathbb{(C)}-1}{M_w} \big\rfloor$. As $ r \leq d_{(P,\pi)}(\mathbb{C}) - 1 $ by  Proposition \ref{phg}, we have 
			$ \max\limits_{J \in  \mathcal{I}^{r}}    \big\{\sum_{i \in J} k_{i}\big\} \leq \max\limits_{J \in   \mathcal{I}^{d_{(P,\pi)}(\mathbb{C}) - 1}}    \big\{\sum_{i \in J} k_{i}\big\}$.  	Thus,  $  N - \lceil log_{q}|\mathbb{C}| \rceil \leq  \max\limits_{J \in   \mathcal{I}^{d_{(P,\pi)}(\mathbb{C}) - 1}}    \big\{  \sum_{i \in J} k_{i}\big\}$.
			Hence $\mathbb{C} $  is MDS with respect to $(P,\pi)$-metric.}
	\end{proof}
\end{theorem}
\begin{theorem}
	If $\mathbb{C} $  is an MDS code with respect to $(P,w)$-metric then $\mathbb{C} $  is MDS  with respect to $P$-metric.
	\begin{proof}
		If $\mathbb{C} $  is  MDS  with respect to $(P,w)$-metric then $  \big\lfloor \frac{d_{(P,w)} \mathbb{(C)}-1}{M_w} \big\rfloor   = n - \lceil log_{q}|\mathbb{C}| \rceil $. By  Proposition \ref{phg}, by considering $k_i = 1 $ for every $i \in [n]$, we have $	\big\lfloor \frac{d_{(P,w)} \mathbb{(C)}-1}{ M_w}\big\rfloor \leq d_{P}(\mathbb{C}) - 1$.  Thus,  $  n - \lceil log_{q}|\mathbb{C}| \rceil \leq d_{P}(\mathbb{C}) - 1$.		Hence $\mathbb{C} $  is MDS code   with respect to $P$-metric.
	\end{proof}
\end{theorem}
\section{Codes on NRT Block Spaces}
Throughout this section, we  consider $P=([n],\preceq)$ to be a chain. Then each ideal in $P$ has a unique maximal element.  Let $v \in \mathbb{F}_{q}^N$, $I_v^{P,\pi}= \langle supp_{\pi}(v) \rangle$ and $j$ be the maximal element of  $I_v^{P,\pi}$. Then,
\begin{equation*}
	w_{(P,w,\pi)}(v)=  \tilde{w}^{k_j}{(v_j)} + (|I_v^{P,\pi}|-1) M_w
\end{equation*} 
where $v_j \in \mathbb{F}_{q}^{k_j}$ and $\tilde{w}^{k_j}{(v_j)}=\max\{w(v_{j_t}) : 1 \leq t \leq k_j\}$.
The  space $ (\mathbb{F}_{q}^N, ~d_{(P,w,\pi)} )$ is called as the  NRT block space (when $P$ is a chain).
\par Let $A_r = \{x= x_{1} \oplus x_{2} \oplus \ldots \oplus x_{n} \in \mathbb{F}_{q}^{N}  : w_{(P,w,\pi)}(x)=r  \}$.  Now, for any $i \in [n]$,  $Max \langle i \rangle = \{i\}$, $| \langle i \rangle | = i$ and  $|\mathcal{I}_{1}^{i}|=1$. Then, 
\begin{enumerate}[label={(\roman*)}]
	\item   For $1 \leq r \leq M_{w} $, we have $v \in A_r$ iff $| \langle supp_{\pi} (v) \rangle | = 1 $ (as $supp_{\pi}(v)=\{1\}$); and hence $	|A_{r}| = |D_{r}^{k_1}|$.
	\item	For $r= tM_{w}+ k $, where $0 < k \leq M_{w}$ and $1 \leq t \leq n-1$, we have $v \in A_r$ iff $| \langle supp_{\pi} (v) \rangle | = t + 1 $. Thus,
	$|A_{tM_{w} +  k }| =  q^{ k_{1} + k_2 + \cdots + k_t} |D_{k}^{k_{t+1}}|$.
\end{enumerate} 
The Singleton bound for  NRT block codes is obtained from Theorem \ref{sbwpic}:
\begin{theorem}\label{chainSingl}
	Let $P$ be a chain and $\mathbb{C} $  be a $(P,w,\pi)$-code of length $N = k_1 + k_2 + \ldots + k_n$ over $\mathbb{F}_q$ with minimum distance $d_{(P,w,\pi)}(\mathbb{C})$. Then $
	\sum\limits_{i \in J} k_{i} \leq N - \lceil log_{q}|\mathbb{C}| \rceil$ for any  ideal $J$ with $|J| = \big\lfloor \frac{d_{(P,w,\pi)} \mathbb{(C)}-1}{M_w} \big\rfloor $.
\end{theorem} 
\begin{corollary}
	If  $k_i=s$ $\forall$ $i \in [n] $ then
	$ \big\lfloor \frac{d_{(P,w,\pi)} \mathbb{(C)}-1}{M_w} \big\rfloor  \leq n - \frac{ \lceil log_{q}|\mathbb{C}| \rceil}{s}$. Moreover, if  $\mathbb{C}$ is a linear $[N, k]$ code   then $ \big\lfloor \frac{d_{(P,w,\pi)} \mathbb{(C)}-1}{M_w} \big\rfloor  \leq  \big\lfloor n - \frac{k}{s}  \big\rfloor $.
\end{corollary}
Let  $r = t M_w+s$ where $s \in  \{1,2,\ldots, M_w\}$ and $t \geq 0$. Let $B_{(P,w,\pi)} (x,~r)$ and $B_{(P,\pi)} (x,~r)$ denote the  $r$-balls centered at $x$ with radius $r$ with respect to the $(P,w,\pi)$-metric and $(P,\pi)$-metric respectively. If $y \in B_{(P,w,\pi)} (x,~r)$ then $w_{(P,w,\pi)}(x-y) \leq tM_w+s $ and  $w_{(P,\pi)}(x-y) = |\langle supp_{\pi} (x-y) \rangle| \leq t+1$. Hence $y \in B_{(P,\pi)} (x,~t+1)$. Thus, we have
\begin{theorem}\label{Packing radius pwpi code lemma}
	Let $ x \in \mathbb{F}_{q}^{N} $ and $w$ be a weight on $\mathbb{F}_{q}$. If $r = t M_w + s$, where $s \in  \{1,2,\ldots, M_w\}$ and $t \geq 0$,   then $B_{(P,w,\pi)} (x,~r) \subseteq B_{(P,\pi)} (x,~t+1)$. Furthermore, $B_{(P,w,\pi)} (x,~r) = B_{(P,\pi)} (x,~t+1)$ if and only if $r$ is a multiple of $M_w$.  
\end{theorem}
The packing radius of a  code $\mathbb{C} $ with respect to any metric $d$ is the
greatest integer $r$ such that the  $r$-balls centered at any two distinct codewords are disjoint. 
\par  Let $R_{(P,\pi)} (\mathbb{C}) = t$ be the packing radius of the code $\mathbb{C} \subseteq \mathbb{F}_{q}^{N} $ and $d_{(P,\pi)} (\mathbb{C})$ be the minimum distance of the code $\mathbb{C}$ with respect to the  $(P,\pi)$-metric.  
Let $R_s =  (t-1) M_w + s$. Then $B_{(P,w,\pi)} (x,~R_s) \subseteq B_{(P,\pi)} (x,~t)$ for each positive integer $ s \leq M_w$ by Theorem \ref{Packing radius pwpi code lemma}. Hence, $R_{(P,w,\pi)} (\mathbb{C}) \geq R_s $ for every $0 < s \leq M_w$.  If $s = M_w$, then  $ R_{(P,w,\pi)} (\mathbb{C}) = R_{M_w} = t M_w $ (again, due to Theorem \ref{Packing radius pwpi code lemma}). Thus, packing radius of a $(P,w,\pi)$-block code $ \mathbb{C}$ is $R_{(P,w,\pi)} (\mathbb{C}) = M_w R_{(P,\pi)} (\mathbb{C}) $. Now, $R_{(P,\pi)} (\mathbb{C}) = d_{(P,\pi)} (\mathbb{C}) -1 $ when $P$ is a chain, (ref.  \cite{nrt block}, Theorem 5). Thus, 
\begin{corollary}[Packing radius]
	The packing radius of a  $(P,w,\pi)$-block code $ \mathbb{C} \subseteq  \mathbb{F}_{q}^{N} $ is $R_{(P,w,\pi)} (\mathbb{C}) = M_w ( d_{(P,\pi)} (\mathbb{C}) - 1 ) $.
\end{corollary}
Now we  determine the relation between the minimum distances of a code $ \mathbb{C} \subseteq  \mathbb{F}_{q}^{N} $ with respect to $(P,w,\pi)$-metric and $(P,\pi)$-metric. 
\begin{theorem}[Minimum distance]\label{min dist}
	Let $\mathbb{C}$ be a linear  $(P,w,\pi)$-code and $m_w = \min\{ w( \alpha) : \alpha \in \mathbb{F}_q \}$ where $P$ is a chain. Then  $d_{(P,w,\pi)} (\mathbb{C}) = m_w + M_w  ( d_{(P,\pi)} (\mathbb{C}) - 1 )$ and $  R_{(P,w,\pi)} (\mathbb{C})= d_{(P,w,\pi)} (\mathbb{C}) - m_w  $.
\end{theorem}
\begin{proof}
	Let $ u \in \mathbb{C}$ be such that $w_{(P,\pi)}(u)=d_{(P,\pi)} (\mathbb{C})$. Let 
	$ w_{(P,\pi)}(u) = i $. As $\mathbb{C}$ is linear, it must be the case that  $\tilde{w}^{k_i}(u_i) = m_w$	where $u_i$ is the $i^{th} $ block of $ u$. We will show that
	$d_{(P,w,\pi)}(C) = w_{(P,w,\pi)} (u)$.
	If $ 0 \neq v \in \mathbb{C}$ and  $w_{(P,w,\pi)}(v) <  w_{(P,w,\pi)}(u)$,    then $   \tilde{w}^{k_j}(v_{j}) + (j-1)M_w <  \tilde{w}^{k_i}(u_{i}) + (i-1)M_w $ where $j$ is the maximum element of $supp_{\pi}(v)$.
	Since $\tilde{w}^{k_j}(v_j ) \geq \tilde{w}^{k_i}(u_i) = m_w$, it follows that $(j -1)M_w <
	(i-1)M_w$ which implies $j < i$. Hence $0 < w_{(P,\pi)}  (v) < w_{(P,\pi)}  (u)$,
	a contradiction.  Therefore,
	$ d_{(P,w,\pi)}  (\mathbb{C})= w_{(P,w,\pi)}(u) = \tilde{w}^{k_i}(u_i) + (i - 1)Mw
	= m_w + (d_{(P,\pi)}  (\mathbb{C}) - 1) M_w$.    
	As $R_{(P,w,\pi)} (\mathbb{C}) = M_w ( d_{(P,\pi)} (\mathbb{C}) - 1 ) $, we have
	$R_{(P,w,\pi)}(\mathbb{C}) = d_{(P,w,\pi)}  (\mathbb{C}) - m_w$. 
\end{proof}
When $P$ is a chain,  Singleton bound of any  $[N, k]$ linear $(P,\pi)$-code $\mathbb{C} $  \cite{bkdnsr}  is given   by  $ \sum_{i \in J} k_{i} \leq  N - \lceil log_{q}|\mathbb{C}| \rceil $ with $ |J|={d_{(P,\pi)}(\mathbb{C}) - 1}$.
 As $ M_w \geq 1$,   $d_{(P,w,\pi)} (\mathbb{C}) - 1 \geq d_{(P,w,\pi)} (\mathbb{C}) - m_w $.   From Theorem \ref{min dist},   $d_{(P,w,\pi)} (\mathbb{C}) - 1 \geq M_w  ( d_{(P,\pi)} (\mathbb{C}) - 1 ) $. Thus,  
 $ \big\lfloor \frac{d_{(P,w,\pi)} \mathbb{(C)}-1}{M_w} \big\rfloor \geq  d_{(P,\pi)} (\mathbb{C}) - 1  $. By combining this with what we have in Proposition \ref{phg}, we get  $ \big\lfloor \frac{d_{(P,w,\pi)} \mathbb{(C)}-1}{M_w} \big\rfloor =  d_{(P,\pi)} (\mathbb{C}) - 1  $. Hence, we have $ \sum\limits_{i \in J} k_{i} \leq N - k$ with $ |J|=\big\lfloor \frac{d_{(P,w,\pi)} \mathbb{(C)}-1}{M_w} \big\rfloor$,
  which is the linear code version of the  Singleton bound obtained in Theorem \ref{chainSingl}. 

\bibliographystyle{amsplain}


\end{document}